\documentclass[a4paper, 10pt, leqno, final]{article}

\usepackage[T1]{fontenc}		     % codifica dei font
\usepackage[utf8]{inputenc}			 % codifica di input
\usepackage[english]{babel}
\usepackage{amsmath}
\usepackage{amssymb}
\usepackage{amsthm}
	\theoremstyle{plain}
		\newtheorem{thm}{Theorem}[section]	% numbered within each
		\newtheorem{cor}[thm]{Corollary}	
		\newtheorem{maincor}{\textsc{Corollary}}		%

		\newtheorem{lem}[thm]{Lemma}		% numbered along with
		\newtheorem{prop}[thm]{Proposition}
	\theoremstyle{definition}
		\newtheorem{defn}[thm]{Definition}	% numbered along with
	\theoremstyle{remark}
		\newtheorem{rem}[thm]{Remark}		% numbered along with
		\newtheorem{note}[thm]{Notation}		% numbered

\numberwithin{equation}{section}	% equation numbering

\usepackage{mathtools}		% per i simboli := e =:
\usepackage{mathrsfs}		% per il math script
\usepackage{eucal}			% per il mathcal + figo

\usepackage{braket}			% per il comando \Set: definizione
\usepackage{mparhack}		% fix bug for \marginpar
\usepackage{relsize}			% dimensioni testo

\usepackage{a4wide}		% aumenta i margini di pagina

\usepackage{booktabs}		% per fare bene le tabelle
\usepackage{multirow}		% per celle su pi righe nelle tabelle
\usepackage{caption}		% per le didascalie di tabelle e figure
\usepackage{rotating}
\usepackage{subfig}

\usepackage{varioref}		% per indicazione di pagina nei riferimenti
\usepackage{footmisc}

\usepackage{graphicx}		% per le figure
\usepackage{epsfig}
\usepackage{enumerate}		% per impostare facilmente gli elenchi numerati
\usepackage[colorlinks=true, linkcolor=black, citecolor=black,
urlcolor=blue]{hyperref}		
\mathtoolsset{showonlyrefs}
\newcommand{\GL}{\mathrm{GL}}

\newcommand{\eht}{\mathbb E_{h}^T}
\newcommand{\eh}{\mathbb E_{h}}
\newcommand{\ud}{u_\delta}
\newcommand{\ima}{\text{Im}}

\newcommand{\trasp}[1]{{#1}^\mathsf{T}}	
\newcommand{\EE}{\mathbf{E}}		% the complex plane
		
\newcommand{\iMor}{\mathrm{n_-}}		% indice di Morse
\newcommand{\iMorse}{\iota_{\scriptscriptstyle{\mathrm{Mor}}}}
\newcommand{\R}{\mathbb{R}}	

\newcommand{\C}{\mathbb{C}}		% the complex plane
\newcommand{\T}{\mathbb{T}}		% the complex plane
\newcommand{\U}{\mathbb{U}}		% gruppo unitario
\newcommand{\OO}{\mathrm{O}}		% gruppo ortogonale
\newcommand{\spec}{\sigma}		% gruppo unitario
\newcommand{\irel}{I}		% indice geometricoù
\newcommand{\ispec}{\iota_{\textup{spec}}}				% indice geometrico
\newcommand{\igeo}{\iota_{\textup{geo}}}		% indice geometrico
\newcommand{\Sp}{\mathrm{Sp}}

\newcommand{\Lagr}{\Lambda}

\newcommand{\Ddt}{\nabla_t}

\newcommand{\rie}[2]{\langle{#1}, {#2} \rangle_g}

\newcommand{\Gr}{\mathrm{Gr}}

\newcommand{\im}{\mathrm{rge}\,}

\newcommand{\ind}{\mathrm{ind}\,}

\newcommand{\ddt}{\mathrm{\dfrac{d}{dt}}}
\let\d\relax
\newcommand{\d}{\mathrm{d}}				% differenziale
\newcommand{\norm}[1]{\left\| #1 \right\|}			% norma
\newcommand{\abs}[1]{\left\lvert #1 \right\rvert}				
\newcommand{\traspinv}[1]{{#1}^\mathsf{-T}}				%

\newcommand{\coiMor}{\coindex}

\renewcommand{\L}{L}	% Lagrangiana
\newcommand{\M}{\mathcal{M}}	%
\newcommand{\N}{\mathbb{N}}		% the natural numbers

\newcommand{\iCLM}{\iota_{\scriptscriptstyle{\mathrm{CLM}}}}

\newcommand{\iomega}[1]{\iota_{#1}}
\newcommand{\Z}{\mathbb{Z}}		% the integer numbers
\newcommand{\coindex}{\mathrm{n_+}}
\newcommand{\iiindex}{\mathrm{n_-}}

\newcommand{\noo}[1]{\overset {\mbox{%
\lower1pt\hbox{${\scriptscriptstyle o}$}}}n^{\mbox{%
\lower2pt\hbox{$\scriptscriptstyle #1$}}}}

\DeclareMathOperator{\spfl}{sf}			% spectral flow
\DeclareMathOperator{\sgn}{sgn}		% signature
\DeclareMathOperator{\sign}{sign}		% signature
\DeclareMathOperator{\rk}{rank}		% rank
\renewcommand{\geq}{\geqslant}

\renewcommand{\tilde}{\widetilde}
\renewcommand{\=}{\coloneqq}			% definisce :=
\newcommand{\email}[1]{\href{mailto:#1}{\textsf{#1}}}

\usepackage{bm}

\newcommand{\Id}{I}
\title{Linear instability  of  periodic orbits of free period Lagrangian systems}
\author{Alessandro Portaluri, Li Wu, Ran Yang \thanks{The author is partially supported by NSFC(N.12001098) and  Doctoral research start-up fund of East China University of Technology(N.DHBK2019204)}}
\date{\today}

\date{\today}
\begin{document}
 \maketitle

\begin{abstract}
In this paper we provide a sufficient condition for the linear instability of a periodic orbit for a free period Lagrangian system on a Riemannian manifold.

 The main result establish a general criterion for the linear instability of a maybe degenerate)  periodic orbit admitting   a orbit cylinder in terms  to the parity of a suitable {\em spectral index\/} encoding the functional and symplectic property of the problem. 
 
\end{abstract}

%%%%%%%%%%%%%%%%%%%%%%%%%%%%%%%%%%%%%%%%%%%%%%%%%%%%%%%%%%%%
%%%%%%%
%%%%%%%
%%%%%%%
%%%%%%%%%%%%%%%%%%%%%%%%%%%%%%%%%%%%%%%%%%%%%%%%%%%%%%%%%%%%

\section{Introduction and description of the problem}\label{sec:intro}

A celebrated result proved by Poincaré in the beginning of the last century put on evidence the relation between  the (linear and exponential) instability of an orientation preserving  closed geodesic as a critical point of the geodesic energy functional on the free loop space on a surface and the minimization property of such a critical point. The literature on this criterion is quite broad. We refer the interested reader to \cite{BJP14, BJP16, HPY19}  and references therein. 

Several years later, the authors in \cite{HS10} proved a generalization of the aforementioned result, by dropping the non-degeneracy assumption and some years later, in \cite{HPY19}, authors provided  a general criterion for detecting the linear instability of a closed geodesic on a finite dimensional semi-Riemannian manifold. 

It is quite natural to understand the role of the energy level $h$ of a periodic orbit for an autonomous Lagrangian system on the instability. Some years ago, in  his expository article \cite{Abb13}, Abobndandolo  studied the question of the existence of periodic orbits of  the free period Lagrangian systems by putting on evidence the role provided  by the Ma\~{n}é critical values. Recently, in \cite{Ure19}, Ure\~{n}a investigated the instability of closed orbits obtained by
minimization for autonomous Lagrangian systems by combining  the  classical principle of Jacobi-Maupertuis with the argument of reduction of order suggested by Carathéodory of reduction of the order. In this way, under some suitable conditions, the dynamics of a free period  Lagrangian system can be represented as the dynamics of a non-autonomous and fixed period Lagrangian system lowering by 1 the degrees of  freedom. 

From a tecnical viewpoint a key step for proving our main result is based on a precise formula relating the Morse indices of a periodic orbit as critical point of the free and fixed time Lagrangian action functional. In \cite{MP10},  authors actually  established such a relation under a non-degenerate assumption of the orbit cylider.

As already observed in   this paper, we generalize the aforementioned results provided in \cite{MP10} by  dropping the non-degeneracy assumption of the orbit cylider by using some new techinques  constructed by the second autho in \cite{Wu16}. Finally, we establish the role provided by the Legendre convexity and Legendre  concavity on the linear stability of the periodic orbit. 

A crucial intermediate step for proving the main result of this paper is based on a spectral flow formula relating  the spectral indices  to a symplectic invariant known in literature as Maslov index (which is an intersection invariant constructed in the Lagrangian Grassmannian manifold of a symplectic space) that plays a crucial role in detecting the stability properties of a periodic orbit.  (For an index formula, we refer the interested reader to  \cite{APS08, MPP05, MPP07, HS09, Por08, RS95} and references therein). Very recently, new spectral flow formulas has been established and applied in the study, for detecting bifurcation of heteroclinic and homoclinic orbits of Hamiltonian systems or bifurcation of semi-Riemannian geodesics. (Cfr. \cite{PW16,MPW17, BHPT19, HP17, HP19, HPY17}).

%%%%%%%%%%%%%%%%%%%%%%%%%%%%%%%%%%%%%%%%%%%%%%%%%%%%%%%%%%%%
%%%%%%%
%%%%%%%
%%%%%%%
%%%%%%%%%%%%%%%%%%%%%%%%%%%%%%%%%%%%%%%%%%%%%%%%%%%%%%%%%%%%

\subsection{Description of the problem and main result}

Let $(M,\rie{\cdot}{\cdot})$ be a smooth $n$-dimensional Riemannian manifold without boundary, which represents the configuration space of a Lagrangian dynamical system.  Elements  in the tangent bundle $TM$ will be  denoted by $(q,v)$, with $q \in M$ and  $v \in T_qM$.  Let  $L: TM\to \R$  be a smooth autonomous (Lagrangian) function  satisfying the following assumptions
\begin{itemize}
	\item[{\bf(N1)\/}] $L$ is  non-degenerate on the fibers of $TM$, that is, for every $(q,v)\in TM$ we have that  $d_{vv}L(q,v)\neq 0$ is non-degenerate as a quadratic form, where $d_{vv}L$ denotes the fiberwise second differential of $L$;
	\item[{\bf(N2)\/}] $L$ is {\em exactly quadratic\/}  in the velocities meaning that the function $ L(q,v)$ is a polynomial of degree at most $2$ with respect to $v$.
\end{itemize}
On the cartesian product $\Lambda^1(M)\times \R^+$ where $\Lambda^1(M)$ denotes the Hilbert manifold of $1$-periodic loops on $M$ having Sobolev regularity $H^1$, we define  the free period Lagrangian action functional given by 
\begin{equation}\label{eq:free period functional-intro}
\mathbb E_h(x,T)\= T\int_0^1 \Big(L\big(x(s),x'(s)/T\big) +h\Big)\, d s,
\end{equation}
where $h$ is a real constant playing the role of energy. 
\begin{defn}
Let $(x,T)$ be a critical point of the Lagrangian action given in \eqref{eq:free period functional-intro}. We term $(x,T)$ {\em non-null\/} if   
\[
\langle d_{vv} L(q,v)x'(t),x'(t) \rangle_g \neq 0 \quad \textrm{ for every } t\in [0,1].
\]
 Moreover, $(x,T)$ is termed 
 \begin{itemize}
 \item  {\em $L$-Positive\/} if $\langle d_{vv} L(q,v)x'(t),x'(t) \rangle_g > 0$;
 \item  {\em $L$-Negative\/} if $\langle d_{vv} L(q,v)x'(t),x'(t) \rangle_g < 0$.
 \end{itemize}
\end{defn}
We recall the classical definition of {\em Legendre convexity\/}
\begin{itemize}
	\item[{\bf(L1)\/}] $L$ is  $\mathscr C^2$ strictly convex  on the fibers of $TM$, that is, for every $(q,v)\in TM$ we get   $d_{vv}L(q,v)> 0$ as a quadratic form.
\end{itemize}
We observe that the $L$-positiveness (resp. negativeness) corresponds to a sort of Legendre convexity (resp. concavity) condition only along the selected orbit. 
Following authors in\cite[Definition 1.2]{MP10} we introduce the  notion of orbit cylinder. 
\begin{defn}\label{def:cylinder}
	A critical point $(x,T)$ of $\eh$  admit an {\em orbit cylinder\/} if there exist $\epsilon>0$ and a smooth (in $s$) family critical points $\{(x_{h+s},T_{h+s}), s\in(-\epsilon,\epsilon)\}$ of $\mathbb{E}_{h+s}$ such that $(x_h,T_h)=(x,T)$. Moreover, this orbit cylinder is called non-degenerate if $T'(h)\neq0$.
\end{defn}
Under the above notation and definitions we are in position to state 
the  main result  of the paper. 

\begin{thm}\label{thm:main-instability}
	Let $(x,T)$ be a non-null periodic orbit of free period Lagrangian action \eqref{eq:free period functional-intro} admitting an  orbit cylinder. If one of the following four statements holds:
	\begin{itemize}
		\item[] { (1)} $x$ is $\bf{L-Positive}$ and  
		\begin{itemize}
			\item[]{\bf (OR)} $x$   is  orientation preserving and $\ispec(x,T) +n$   is even
			\item[]{\bf (NOR)} $x$   is non orientation preserving and $\ispec(x,T) +n$  is odd
		\end{itemize}
		\item[] { (2)} $x$ is $\bf{L-Negative}$ and
		\begin{itemize}
			\item[]{\bf (OR)} $x$   is  orientation preserving and $\ispec(x,T) +n$   is odd
			\item[]{\bf (NOR)} $x$   is non orientation preserving and $\ispec(x,T) +n$  is even
		\end{itemize}	
	\end{itemize}	
	then $x$ is linearly unstable.  
\end{thm}
We observe that in the $L$-positive case, the spectral index actually reduces to the Morse index whilst in the $L$-negative  case, it  reduces to the Morse co-index. A straightforward consequence is the provided by the following result. 
\begin{maincor}\label{cor:instability for P positive}
	Let us assume that $(x,T)$ is a periodic orbit of free period Lagrangian action  \eqref{eq:free period functional-intro} admitting an  orbit cylinder  and we  assume that (L1) holds. If one of the following two alternatives holds
	\begin{itemize}
		\item[]{\bf (OR)} $x$   is  orientation preserving and $\iMorse(x,T) +n$   is even
		\item[]{\bf (NOR)} $x$   is non orientation preserving and $\iMorse(x,T) +n$  is odd
	\end{itemize}
	then $x$ is linearly unstable.
\end{maincor}
The paper is organized as follows

\tableofcontents

\subsection*{Notation}
At last, for the sake of the reader, let us introduce some common notation that we shall use henceforth throughout the paper. 
\begin{itemize}
	\item $(M,\rie{\cdot}{\cdot})$ denotes a Riemannian manifold without boundary; $TM$ its  tangent bundle and  $T^*M$ its cotangent bundle.
	\item $\Lambda^1(M)$ is the Hilbert manifold of loops on manifold $M$ having Sobolev regularity $H^1$
	\item $\omega$ denotes the symplectic structure $J$ the standard  symplectic matrix such that $\omega(\cdot, \cdot)=\langle J \cdot, \cdot \rangle$ where $\langle \cdot, \cdot \rangle$
	denotes the standard Euclidean product  
	\item $\iMorse(x)$ stands for the Morse index of $x$, $\ispec(x)$ for the spectral index of $x$, $\igeo(x)$: for the geometrical index index of $x$, $\iomega{1}$ denotes the Maslov-type index or Conley-Zehnder index of a symplectic matrix path, $\iCLM$ denotes the Maslov (intersection) index and finally $\spfl$ denotes the  spectral flow
	\item $P(L)$ denotes the set of $T$-periodic solutions of the Euler-Lagrange Equation ; $P(H)$: the set of $T$-periodic solutions of the Hamiltonian Equation
	\item $\delta_{ij}$ is the Kronecker symbol.  $\Id_X$ or just $\Id$ will denote the identity operator on a space $X$ and we set for simplicity  $\Id_k := \Id_{\R^k}$ for $k \in \N$.  $\Gr(\cdot)$ denotes the  graph of  its argument; $\Delta$ denotes the graph of identity matrix $\Id$
	\item $\U$ is  the unit circle of the complex plane
	\item $\OO(n)$ denotes the orthogonal group $\Sp(2n,\R)$ or just $\Sp(2n)$ denotes the $2n\times 2n$ real symplectic group
	\item $\mathfrak P$: linearized Poincar\'{e} map
	\item $\mathcal{BF}^{sa}$: the set of all bounded	selfadjoint Fredholm operators,  $\sigma(\cdot )$ denotes the spectrum of the operator in its argument
	\item We denote throughout by the symbol $\trasp{\cdot}$ (resp. $\traspinv{\cdot}$) the transpose (resp. inverse transpose) of the operator $\cdot$. Moreover $\im(\cdot), \ker(\cdot)$ and $\rk(\cdot)$ denote respectively the image, the kernel and the rank of the argument
	\item $\Gamma$ denotes the crossing form and  $\coiMor$/$\iMor$ denote respectively the dimensions  of the positive/negative  spectral  spaces and finally $\sgn(\cdot)$ is  the  signature of the quadratic form (or matrix) in its argument and it is  given by $\sgn(\cdot)=\coiMor(\cdot)-\iMor(\cdot)$ 
\end{itemize}

\subsection*{Acknowledgements}
The third named author wishes to thank Prof. Xijun Hu for providing excellent working conditions during his stay in Shandong University and some useful discussions in orbit cylinder.

%%%%%%%%%%%%%%%%%%%%%%%%%%%%%%%%%%%%%%%%%%%%%%%%%%%%%%%%%%%%
%%%%%%%
%%%%%%%
%%%%%%%
%%%%%%%%%%%%%%%%%%%%%%%%%%%%%%%%%%%%%%%%%%%%%%%%%%%%%%%%%%%%

\section{Lagrangian dynamics and Variational framework}\label{sec:variational-framework}

In this preliminary section we fix our basic notation and we explicitly provided the computation of the first and second variation of the free period Lagrangian action functional.

%%%%%%%%%%%%%%%%%%%%%%%%%%%%%%%%%%%%%%%%%%%%%%%%%%%%%%%%%%%%
%%%%%%%
%%%%%%%
%%%%%%%
%%%%%%%%%%%%%%%%%%%%%%%%%%%%%%%%%%%%%%%%%%%%%%%%%%%%%%%%%%%%

\subsection{Free-period Lagrangian action }

Let $(M,\rie{\cdot}{\cdot})$ be a (not necessarily compact or connected) smooth $n$-dimensional Riemannian manifold without boundary, which represents the configuration space of a Lagrangian dynamical system and we denote by $\norm{\cdot}$ the Riemannian norm.   Elements  in the tangent bundle $TM$  will be  denoted by $(q,v)$, with $q \in M$ and  $v \in T_qM$. The metric $\rie{\cdot}{\cdot}$  induces a metric on $TM$, Levi-Civita connections both on $M$ and $TM$ as well as the isomorphisms 
\begin{equation}
T_{(q,v)}TM= T_{(q,v)}^hTM\oplus T_{(q,v)}^vTM\cong T_qM \oplus T_qM,
\end{equation}
where $T^v_{(q,v)}TM= \ker D\tau(q,v)$ and  $\tau: TM \to M$ denotes the canonical projection. 
\begin{note}
	We shall denote by  $\Ddt$ the covariant derivative  of a vector field along a smooth curve $x$ with respect to the metric $\rie{\cdot}{\cdot}$. $\partial_q$ (resp. $\partial_v$) denotes the partial derivative along the horizontal part  (resp.  vertical part) given by the Levi-Civita connection in the splitting of $TTM$ and We shall denote by $\partial_{vv}, \partial_{qv}, \partial_{qq}$ the  components of the Hessian in the splitting of $TTM$. 
\end{note}
%
%Let  $L:TM\to \R$  be a smooth autonomous (Lagrangian) function  satisfying the following assumptions
%\begin{itemize}
%	\item[{\bf(N1)\/}] $L$ is  non-degenerate on the fibers of $TM$, that is, for every $(q,v)\in TM$ we have that  $d_{vv}L(q,v)\neq 0$ is non-degenerate as a quadratic form,
%	\item[{\bf(N2)\/}] $L$ is {\em exactly quadratic\/}  in the velocities meaning that the function $ L(q,v)$ is a polynomial of degree at most $2$ with respect to $v$.
%\end{itemize}
%
%{\color{red}From now on we would like to study the Lagrangian action on the space of closed curves
%of arbitrary period. 
Given a positive number $T\in \R$ let us denote by $\T$ the one-dimensional torus $\T=\R/T\Z$. Let $\Lambda^1_T(M)$ be the  Hilbert manifold of all loops $y: \T \to M$ having  Sobolev class $H^1$. Setting  $x(t)=y(tT), t\in[0,1]$ we get that the closed curve $y$ will be identified with the pair $(x,T)$. The action of $y$ on time interval $[0,T]$ is given by 
\begin{equation}\label{eq:original functional}
 \int_0^T [L\big(y(s),y'(s)\big) +h]\, d s= T\int_0^1 [L\big(x(t),x'(t)/T\big) +h]\, d t,
\end{equation}
where  $h\in \R$ is a parameter. For convenience, we denote 
\begin{equation}\label{eq:free period functional}
\mathbb E_h(x,T)\= T\int_0^1 [L\big(x(t),x'(t)/T\big) +h]\, d t.
\end{equation}
In short-hand notation, in what follows we use $\Lambda^1(M)$ instead of  $\Lambda^1_1(M)$. In this way, we define a one-to-one correspondente between $\bigcup_{T>0}\Lambda^1_T(M) $ and $\Lambda^1(M)\times \R^+$ which preserves the action values. Bearing in mind such a correspondente,  the free period action functional \eqref{eq:free period functional}  is defined on the manifold $\Lambda^1(M)\times \R^+$.

It is well-known that the tangent space  $T_x\Lambda^1(M)$  can be identified in a natural way with the Hilbert space of $1$-periodic $H^1$ (tangent) vector fields along $x$, i.e. 
\begin{equation}\label{eq:hilbert-space}
\mathcal H(x)=\Set{\xi \in H^1(\R/\Z, TM)| \tau \circ \xi=x}.
\end{equation}
It is worth noticing that  under the  (N2) assumption, the Lagrangian functional  is of regularity class $\mathscr C^2$ (actually it is smooth). Let $\langle\langle \cdot, \cdot \rangle \rangle_1$ denote the Riemannian metric on $\Lambda^1(M)$ defined by 
\begin{equation}\label{eq:productonloopspace}
\langle\langle \xi, \eta\rangle \rangle_1\= \int_0^1	 \big[\rie{\Ddt\xi}{\Ddt \eta} +\rie{\xi}{\eta}\big]\, dt,  \qquad  \forall\, \xi, \eta \in\mathcal H(x).
\end{equation}

For $(\xi,b)\in \mathcal H(x)\times \R$, the first variation of $\mathbb E_h$ at $(x,T)\in\Lambda^1(M)\times \R^+$ is given by
\begin{multline}\label{eq:first-variation-free-1}
d\mathbb E_h(x,T)[(\xi,0)]= T\int_0^1\big[d_q L\big(x(t), x'(t)/T\big)[\xi] + d_v L\big(x(t), x'(t)/T\big)[\Ddt\xi/T]\big]\, dt	\\= 
\int_0^1 \Big[\langle T\partial_q L\big(x(t), x'(t)/T\big), \xi \rangle_g  +\big\langle  \partial_v L\big(x(t), x'(t)/T\big),\Ddt\xi\big\rangle_g \Big]\, dt\\= \int_0^1 \left[\left\langle T\partial_q L\big(x(t), x'(t)/T\big)- \dfrac{D}{dt}\partial_v L\big(x(t), x'(t)/T\big),\xi\right \rangle_g \right]\, dt\\+ 
\langle  \partial_v L\big(x(t), x'(t)/T\big),\xi\rangle_g\Big\vert_{0}^1.
\end{multline} 
For $(0,b)\in \mathcal H(x)\times \R$, the first variation of $\mathbb E_h$ at $(x,T)\in\Lambda^1(M)\times \R^+$ is given by  
\begin{multline}\label{eq:first-variation-free-2}
d\mathbb E_h(x,T)[(0,b)]=\int_0^1\Big[(L(x(t),x'(t)/T)+h-d_vL(x(t),x'(t)/T)[x'(t)/T]) \cdot b \Big]dt\\=\int_0^1\Big[(L(x(t),x'(t)/T)+h-\langle\partial_vL(x(t),x'(t)/T),x'(t)/T\rangle_g) \cdot b \Big] dt.
\end{multline} 
By equations \eqref{eq:first-variation-free-1} and  \eqref{eq:first-variation-free-2},  up  to standard regularity arguments we get that critical points $(x,T)$  of $\mathbb E_h$ are $1$-periodic solutions of corresponding Euler-Lagrange equation having energy $h$. So,  $(x,T)$ satisfies the following equations:
\[\label{eq:E-L equation-free}
\begin{cases}
\dfrac{D}{dt} \Big(\partial_v L\big(x(t), x'(t)/T\big)\Big)	= T\partial_q L\big(x(t), x'(t)/T\big), \qquad t \in (0,1)\\[12pt]
L(x(t),x'(t)/T)+h-\langle\partial_vL(x(t),x'(t)/T),x'(t)/T\rangle_g=0
\end{cases}
\]
Now, being $\mathbb E_h$ smooth it follows  that the first variation  $d\mathbb E_h(x,T)$  at $(x,T)$ coincides with the Fréchét differential $D\mathbb E_h(x,T)$ and  the second variation of $\mathbb E_h$ at $(x,T)$ coincides with  $D^2 \mathbb E_h(x,T)$. By  a straightforward computation, we get 
\begin{multline}\label{eq:second-variation-free-1}
d^2 \mathbb E_h(x,T)[(\xi,0),(\eta,0)]= \int_0^1 \Big\{\frac 1 Td_{vv} L\big(x(t), x'(t)/T\big)[\Ddt\xi, \Ddt\eta]+ d_{qv} L\big(x(t), x'(t)/T\big)[\xi, \Ddt\eta]\\+ d_{vq} L\big(x(t), x'(t)/T\big)[\Ddt\xi, \eta]+ Td_{qq} L\big(x(t), x'(t)/T\big)[\xi, \eta]\Big\} dt\\=
\int_0^1 \big\langle\frac 1 T\partial_{vv} L\big(x(t), x'(t)/T\big)\Ddt\xi, \Ddt\eta\big\rangle_g+ \langle \partial_{qv} L\big(x(t), x'(t)/T\big),\xi, \Ddt\eta\rangle_g\\+ \langle \partial_{vq} L\big(x(t), x'(t)/T\big)\Ddt\xi, \eta\rangle_g+ \langle T\partial_{qq} L\big(x(t), x'(t)/T\big)\xi, \eta\rangle_g dt.
\end{multline}
\begin{equation}\label{eq:second-variation-free-2}
d^2 \mathbb E_h(x,T)[(0,b),(0,d)]=\int_0^1\frac{1}{T^3}\langle \partial_{vv} L\big(x(t), x'(t)/T\big)x'(t), x'(t)\rangle_g\cdot bd\ dt.
\end{equation}
\begin{multline}\label{eq:second-variation-free-3}
d^2 \mathbb E_h(x,T)[(\xi,0),(0,d)]=  \int_0^1\Big\{- \frac{1}{T^2}d_{vv} L\big(x(t), x'(t)/T\big)[x'(t), \Ddt\xi]\cdot d \\+(d_qL(x(t),x'(t)/T)-\frac 1 Td_{qv}L(x(t),x'(t)/T)x'(t))[\xi]\cdot d\Big\}\ dt\\=
\int_0^1\Big\{- \frac{1}{T^2}\langle \partial_{vv} L\big(x(t), x'(t)/T\big)x'(t), \Ddt\xi\rangle_g\cdot d \\+\langle \partial_qL(x(t),x'(t)/T)-\frac 1 T\partial_{qv}L(x(t),x'(t)/T)x'(t),\xi\rangle_g\cdot d\Big\} dt.
\end{multline}
\begin{multline}\label{eq:second-variation-free-4}
d^2 \mathbb E_h(x,T)[(0,b),(\eta,0)]=\int_0^1 \Big\{-\frac{1}{T^2}d_{vv} L\big(x(t), x'(t)/T\big)[x'(t), \Ddt\eta]\cdot b\\  +(d_qL(x(t),x'(t)/T)-\frac 1 Td_{vq}L(x(t),x'(t)/T)x'(t))[\eta]\cdot b\Big\} dt\\=
\int_0^1\Big\{  -\frac{1}{T^2}\langle \partial_{vv} L\big(x(t), x'(t)/T\big)x'(t), \Ddt\eta\rangle_g\cdot b \\+\langle \partial_qL(x(t),x'(t)/T)-\frac 1 T\partial_{vq}L(x(t),x'(t)/T)x'(t),\eta\rangle_g\cdot b \Big\}dt.
\end{multline}
Summing up, we get 
\begin{multline}\label{eq:second-variation-free}
d^2 \mathbb E_h(x,T)[(\xi,b),(\eta,d)]=
\int_0^1 \big\langle\frac 1 T\partial_{vv} L\big(x(t), x'(t)/T\big)\Ddt\xi, \Ddt\eta\big\rangle_g+ \langle \partial_{qv} L\big(x(t), x'(t)/T\big)\xi, \Ddt\eta\rangle_g\\+ \langle \partial_{vq} L\big(x(t), x'(t)/T\big)\Ddt\xi, \eta\rangle_g+ \langle T\partial_{qq} L\big(x(t), x'(t)/T\big)\xi, \eta\rangle_g dt\\+\int_0^1\big\{  -\frac{1}{T^2}\langle \partial_{vv} L\big(x(t), x'(t)/T\big)x'(t), \Ddt\eta\rangle_g\cdot b- \frac{1}{T^2}\langle \partial_{vv} L\big(x(t), x'(t)/T\big)x'(t), \Ddt\xi\rangle_g\cdot d \\+\langle \partial_qL(x(t),x'(t)/T)-\frac 1 T\partial_{qv}L(x(t),x'(t)/T)x'(t),\xi\rangle_g\cdot d \\+\langle \partial_qL(x(t),x'(t)/T)-\frac 1 T\partial_{vq}L(x(t),x'(t)/T)x'(t),\eta\rangle_g\cdot b\\+\frac{1}{T^3}\langle \partial_{vv} L\big(x(t), x'(t)/T\big)x'(t), x'(t)\rangle_g\cdot bd\big\}dt.
\end{multline}
We set
\begin{multline}
\bar P(t)\= \partial_{vv} L\big(x(t), x'(t)/T\big),  \qquad \bar Q(t)\= \partial_{qv} L\big(x(t), x'(t)/T\big), \qquad \trasp{\bar{Q}}(t)\= \partial_{vq} L\big(x(t), x'(t)/T\big)\\
\bar R(t)\= \partial_{qq} L\big(x(t), x'(t)/T\big),  \qquad \partial_qL(t)=\partial_qL(x(t),x'(t)/T),\qquad \kappa(t)=\langle \bar P(t)x'(t),x'(t)\rangle_g
\end{multline}
Under the above notation and integrating by  parts in the second variation, we get  
\begin{multline}\label{eq:index-form-on-manifold-free} 
d^2 \mathbb E_h[(\xi,b),(\eta,d)]=\int_0^1 \big\langle\frac 1 T\bar P(t)\Ddt\xi, \Ddt\eta\big\rangle_g+ \langle \bar Q(t)\xi, \Ddt\eta\rangle_g+ \langle \trasp{\bar{Q}}(t)\Ddt\xi, \eta\rangle_g+ \langle T\bar R(t)\xi, \eta\rangle_g dt\\+\int_0^1\big\{  -\frac{1}{T^2}\langle\bar P(t)x'(t), \Ddt\eta\rangle_g\cdot b- \frac{1}{T^2}\langle \bar P(t)x'(t), \Ddt\xi\rangle_g\cdot d +\langle \partial_qL(t)-\frac 1 T\bar Q(t)x'(t),\xi\rangle_g\cdot d \\+\langle \partial_qL(t)-\frac 1 T\trasp{\bar{Q}}(t)x'(t),\eta\rangle_g\cdot b+\frac{1}{T^3}\kappa(t)\cdot bd\big\}dt
\\=\int_0^1\langle -\dfrac{D}{dt}\big[\frac 1 T\bar{P}(t)\nabla_t \xi+ \bar{Q}(t)\xi\big]+ \trasp{\bar{Q}}(t)\Ddt\xi+T\bar{R}(t)\xi,\eta\rangle_g dt+\big[\langle \frac 1 T\bar{P}(t)\Ddt \xi+ \bar{Q}(t)\xi,\eta\rangle_g\big]^1_{t=0}\\
+\int_0^1\big\{  -\frac{1}{T^2}\langle\bar P(t)x'(t), \Ddt\eta\rangle_g\cdot b- \frac{1}{T^2}\langle \bar P(t)x'(t), \Ddt\xi\rangle_g\cdot d +\langle \partial_qL(t)-\frac 1 T\bar Q(t)x'(t),\xi\rangle_g\cdot d \\+\langle \partial_qL(t)-\frac 1 T\trasp{\bar{Q}}(t)x'(t),\eta\rangle_g\cdot b+\frac{1}{T^3}\kappa(t)\cdot bd\big\}dt.
\end{multline}
For a given critical point $(x,T)$ of $E_h$, let us  consider the following fixed period  Lagrangian action functional:
\begin{equation}\label{eq:fixed period functional}
\eht(x)\= T\int_0^1 [L\big(x(t),x'(t)/T\big) +h]\, dt,
\end{equation}
where $x\in \Lambda^1(M)$. Actually,  $\eht$ is the restriction of $\eh$ to the submanifold $\Lambda^1(M)\times\{T\}$. By similar calculations, the first variation of $\eht$ at $x\in\Lambda^1(M)$ is given by 
\begin{multline}\label{eq:first-variation-fixed}
d\eht(x)[\xi]= T\int_0^1\big[d_q L\big(x(t), x'(t)/T\big)[\xi] + d_v L\big(x(t), x'(t)/T\big)[\Ddt\xi/T]\big]\, dt\\= 
\int_0^1 \Big[\langle T\partial_q L\big(x(t), x'(t)/T\big), \xi \rangle_g  +\big\langle  \partial_v L\big(x(t), x'(t)/T\big),\Ddt\xi\big\rangle_g \Big]\, dt\\= \int_0^1 \left[\left\langle T\partial_q L\big(x(t), x'(t)/T\big)- \dfrac{D}{dt}\partial_v L\big(x(t), x'(t)/T\big),\xi\right \rangle_g \right]\, dt
\\+ 
\langle \partial_v L\big(x(t), x'(t)/T\big),\xi\rangle_g\Big\vert_{0}^1.
\end{multline}
Critical points $x$  of $\eht$ are $1$-periodic solutions of corresponding Euler-Lagrange Equation:
\begin{equation}\label{eq:E-L equation-fixed}
\dfrac{D}{dt} \Big(\partial_v L\big(x(t), x'(t)/T\big)\Big)	= T\partial_q L\big(x(t), x'(t)/T\big), \qquad t \in (0,1).
\end{equation}
The second variation of $\eht$ at $x$ given by 
\begin{multline}\label{eq:second-vari-fixed}
d^2 \eht(x)[\xi,\eta]=\int_0^1 \big\langle\frac 1 T\bar P(t)\Ddt\xi, \Ddt\eta\big\rangle_g+ \langle \bar Q(t)\xi, \Ddt\eta\rangle_g+ \langle \trasp{\bar{Q}}(t)\Ddt\xi, \eta\rangle_g+ \langle T\bar R(t)\xi, \eta\rangle_g dt
\\=\int_0^1\langle -\dfrac{D}{dt}\big[\frac 1 T\bar{P}(t)\nabla_t \xi+ \bar{Q}(t)\xi\big]+ \trasp{\bar{Q}}(t)\Ddt\xi+T\bar{R}(t)\xi,\eta\rangle_g dt+\big[\langle \frac 1 T\bar{P}(t)\Ddt \xi+ \bar{Q}(t)\xi,\eta\rangle_g\big]^1_{0}.
\end{multline}
By Equation \eqref{eq:second-vari-fixed} we get   that $ \xi \in \ker d^2 \eht(x)$  if and only if $\xi$ is a $H^2$ vector field along $x$ which solves weakly (in the Sobolev sense)  the following boundary value problem
\begin{equation}\label{eq:Sturm-manifolds}
\begin{cases}
-\dfrac{D}{dt}\big(\dfrac 1 T\bar{P}(t)\nabla_t \xi+ \bar{Q}(t)\xi\big)+ \trasp{\bar{Q}}(t)\Ddt\xi+T\bar{R}(t)\xi=0, \qquad t \in (0,1)\\
\\
\xi(0)=\xi(1), \qquad \nabla_t \xi(0)= \nabla_t\xi(1).
\end{cases}	
\end{equation}
By standard bootstrap arguments, it follows that $\xi$ is also a classical (smooth) solution of Equation~\eqref{eq:Sturm-manifolds}. 
\begin{rem}
	It is easy to prove that $x$ is a critical point of fixed period Lagrangian system \eqref{eq:fixed period functional} provided that $(x,T) $ is a critical point of free period Lagrangian system.
\end{rem}
The next result answer the question about the existence of an orbit cylinder about a $T$-periodic orbit $x$. We refer the interested reader to \cite[Section 4.1, Proposition 2]{HZ11} for the proof. 
\begin{prop}\label{prop:hofer-zehnder}
	Assume a periodic solution $x(t,E^*)$ of a Hamiltonian vector field $X_H$ on $M$ having energy $E^*= H(x(t,E^*))$ and period $T^*$ has exactly two Floquet multipliers equal to
	$1$. Then there exists a unique and smooth $1$-parameter family $x(t,E)$ of periodic solutions with periods $T(E)$ close to $T^*$ and lying on the energy surfaces
	$H(x(t,E))= E$ for $|E-E^*|$ sufficiently small. Moreover, $T(E)\rightarrow T(E^*)$ as $E\rightarrow E^*$.	
\end{prop}

\section{Spectral index, geometrical index and Poincaré map}\label{subsec:ortho-triv}

The goal of this section is to associate at the critical point $(x,T)$ of the free period Lagrangian action \eqref{eq:free period functional} and to a critical point $x$ of the fixed period Lagrangian action \eqref{eq:fixed period functional} the  {\em spectral indices\/}, defined in terms of the spectral  flow of a suitable path of Fredholm quadratic forms. We refer the interested reader to Appendix~\ref{sec:spectral-flow} and references therein for a discussion about the spectral flow.

%%%%%%%%%%%%%%%%%%%%%%%%%%%%%%%%%%%%%%%%%%%%%%%%%%%%%%%%%%%%
%%%%%%%
%%%%%%%
%%%%%%%
%%%%%%%%%%%%%%%%%%%%%%%%%%%%%%%%%%%%%%%%%%%%%%%%%%%%%%%%%%%%

\subsection{Spectral index: an intrinsic (coordinate free) definition}

Given $(x,T)$ be a critical point of $\eh$,   for any  $s \in [0, +\infty)$ we let   $\mathcal I_s: (\mathcal H(x)\times \R )\times (\mathcal H(x)\times \R) \to \R$ the bilinear form defined  by 
\begin{multline}\label{eq:path-index-free}
\mathcal I_s[(\xi,b),(\eta,d)]\=d^2 \eh(x,T)[(\xi,b),(\eta,d)]+ s \alpha(x,T)[(\xi,b),(\eta,d)]\\= 
\int_0^1 \{\big\langle\frac 1 T\bar P(t)\Ddt\xi, \Ddt\eta\big\rangle_g+ \langle \bar Q(t)\xi, \Ddt\eta\rangle_g+ \langle \trasp{\bar{Q}}(t)\Ddt\xi, \eta\rangle_g+ \langle T\bar R(t)\xi, \eta\rangle_g\} dt\\+\int_0^1\big\{  -\frac{1}{T^2}\langle\bar P(t)x'(t), \Ddt\eta\rangle_g\cdot b- \frac{1}{T^2}\langle \bar P(t)x'(t), \Ddt\xi\rangle_g\cdot d +\langle \partial_qL(t)-\frac 1 T\bar Q(t)x'(t),\xi\rangle_g\cdot d \\+\langle \partial_qL(t)-\frac 1 T\trasp{\bar{Q}}(t)x'(t),\eta\rangle_g\cdot b+\frac{1}{T^3}\kappa(t)\cdot bd\big\}dt + s \alpha(x,T)[(\xi,b),(\eta,d)]\\
\textrm{ where } \alpha(x,T)[(\xi,b),(\eta,d)]\=\int_0^1 \{\langle \frac 1 T\bar P(t)\xi, \eta \rangle_g+\frac{1}{T^3}\kappa(t)bd\}\, dt.
\end{multline}
\begin{note}
	In short-hand notation and if no confusion can arise, we set $ \mathcal Q^h\= d^2 \eh(x,T)$.
\end{note}
\begin{prop}\label{thm:famiglia-Fredholm-free}
	For any $s \in [0, +\infty)$ let $\mathcal Q_s$ denote the quadratic form associated to  $\mathcal I_s$ defined in Equation~\eqref{eq:path-index-free}. Then 
	\begin{enumerate}
		\item[] $s\mapsto \mathcal Q_s$ is a smooth  path of Fredholm quadratic forms onto $\mathcal H(x)\times \R$.
		In particular $ \mathcal Q^h$ is a Fredholm quadratic form on $\mathcal H(x)\times \R$.
	\end{enumerate}
\end{prop}
\begin{proof}
	Similar to discussions in \cite[Proposition 3.2]{PWY21} and be omitted.
	\end{proof}

Assume $(x,T)$ is a critical point of free period Lagrangian system \eqref{eq:free period functional}, then we can define the bilinear form  $\mathcal I_{s,T}: \mathcal H(x)\times \mathcal H(x) \to \R$ of $x$ as a critical point of system \eqref{eq:fixed period functional} in the same way as $(x,T)$. In fact, $\mathcal I_{s,T}$ is given by 
\begin{multline}\label{eq:path-index-fixed}
\mathcal I_{s,T}[\xi,\eta]\=d^2 \eht(x)[\xi,\eta]+ s \alpha_T(x)[\xi,\eta]\\= 
\int_0^1 \{\big\langle\frac 1 T\bar P(t)\Ddt\xi, \Ddt\eta\big\rangle_g+ \langle \bar Q(t),\xi, \Ddt\eta\rangle_g+ \langle \trasp{\bar{Q}}(t)\Ddt\xi, \eta\rangle_g+ \langle T\bar R(t)\xi, \eta\rangle_g\} dt + s \alpha_T(x)[\xi,\eta]\\
\textrm{ where } \alpha_T(x)[\xi,\eta]\=\int_0^1 \{\langle \frac 1 T\bar P(t)\xi,\eta \rangle_g\, dt.
\end{multline}
Set $ \mathcal Q^h_T\= d^2 \eht(x)$, similar to Proposition \ref{thm:famiglia-Fredholm-free} it can be proved $\mathcal I_{s,T}$ and $\mathcal Q^h_T$ both are Fredholm quadratic forms. For any $s \in [0, +\infty)$ let $\mathcal Q_{s,T}$  denote the quadratic form associated to  $\mathcal I_{s,T}$. 

Now we are entitled to give the following definition of spectral indices. 

\begin{defn}\label{def:spectral-index}
	Let $(x,T)$ be a non-null critical point of free period Lagrangian action \eqref{eq:free period functional}. We term {\em spectral indices  of $(x,T)$\/} and $x$ respectively the integers 
	\begin{eqnarray}\label{eq:spectral-index-manifold}
	\ispec(x,T)\= \spfl\big(\mathcal Q_s, s \in [0, s_0]\big)\quad \textrm{ and } \quad 
	\ispec^T(x)\= \spfl\big(\mathcal Q_{s,T}, s \in [0, s_0]\big).
	\end{eqnarray}
	where the (RHS) denotes the spectral flow of the path of Fredholm quadratic forms for a large enough $s_0 >0$. 
	\end{defn}
\begin{rem}
	It is worth noticing that Definition \ref{def:spectral-index} is well-posed in the sense that it doesn't depend upon $s_0$. This fact will be proved in the sequel and it is actually a direct consequence of Lemma \ref{thm:non-degenerate-s_0-forme-free}. 
\end{rem}
\begin{prop}\label{thm:spectral-index-morse-index}
	We assume that (L1) holds. 
%	$\mathscr C^2$-strictly convex on $TM$, meaning that there exists $\ell_0>0$ such that 
%	\[
%	\partial_{vv}L\big(q,v\big) \geq \ell_0 \Id
%	\]
%	for every $(q,v)\in TM$, then 
	Then the Morse indices of $(x,T)$ and $x$ (i.e. the dimension of the maximal negative subspace of the Hessian of $\mathcal Q^h$ and $\mathcal Q^{h_T}$) are both finite and the following equalities hold
	\[
	\ispec(x)= \iMorse(x,T)\quad \textrm{ and } \quad  \ispec^T(x)= \iMorse(x)
	\]
\end{prop}
\begin{proof}
	We will only consider the free period case since the other is completely analogous. We start observing that if $L$ is $\mathscr C^2$-strictly convex on  $TM$, then $\mathcal Q^h$ is a  positive Fredholm quadratic form and hence $\mathcal Q_s$ is a path of essentially positive Fredholm quadratic forms (being a weakly compact perturbation of a positive definite quadratic form). In particular the  Morse index of $\mathcal Q_s$ is  finite  for every $s \in [0,+\infty)$. If $s_0$ is large enough, the form $\mathcal Q_{s_0}$ is non-degenerate and begin also semi-positive is actually positive definite. Thus its Morse index vanishes. Now, since for path of essentially positive Fredholm quadratic it holds that 
	\[
	\spfl(\mathcal Q_s, s\in [0,s_0])=\iMorse (\mathcal Q_0)- \iMorse (\mathcal Q_{s_0})= \iMorse (\mathcal Q_0)
	\]
	the conclusion readily  follows. 
\end{proof}

%%%%%%%%%%%%%%%%%%%%%%%%%%%%%%%%%%%%%%%%%%%%%%%%%%%%%%%%%%%%
%%%%%%%
%%%%%%%
%%%%%%%
%%%%%%%%%%%%%%%%%%%%%%%%%%%%%%%%%%%%%%%%%%%%%%%%%%%%%%%%%%%%

\subsection{Pull-back bundles and push-forward of Fredholm forms}

We denote by $\mathcal E$ the $\rie{\cdot}{\cdot}$-orthonormal and parallel frame $\EE$, pointwise given by 
\[
\mathcal E(t)=\{e_1(t), \dots, e_n(t)\} 
\]
and, if $(x,T)$ is a critical point,  we let $\bar A: T_{x(0)}M\to T_{x(1)}M\cong T_{x(0)}M$ the $\rie{\cdot}{\cdot}$-orthogonal operator defined by 
\[
\bar A e_j(0)= e_j(1).
\] 
Such a frame $\EE$, induces a trivialization of the pull-back bundle $x^*(TM)$ over $[0,1]$ through the smooth curve $x:[0,1] \to M$; namely  the  smooth one parameter family  of isomorphisms  
\begin{multline}\label{eq:parallel-frame}
[0,1] \ni t \longmapsto E_t \quad \textrm{ where } \quad E_t: \R^n \ni e_i \longmapsto e_i(t) \in  T_{x(t)}M\qquad \forall\ t \in [0,1] \textrm{ and } i =1, \ldots, n\\
\textrm{  are such that } \langle E_t e_i, E_t e_j\rangle_g=  \delta_{ij} \textrm { and } \nabla_t E_t e_i=0
\end{multline}
here $\{e_i\}_{i=1}^n$ is the canonical basis of $\R^n$ and  $\delta_{ij}$ denotes the Kronecker symbol.  

By Equation~\eqref{eq:parallel-frame} we get that the pull-back by $E_t$ of the metric $\rie{\cdot}{\cdot}$ induces  the Euclidean product on $\R^n$ and moreover this pull-back is independent on $t$, as directly follows by the orthogonality assumption on the frame $\EE$.

We set $A\=E_0^{-1}\bar A^{-1} E_1 \in \OO(n)$ and define
\begin{equation}\label{eq:Ad}
A_d\= \begin{bmatrix}
A & 0 \\ 0 & A
\end{bmatrix}.
\end{equation} 
Let us now consider the Hilbert space 
\begin{equation}\label{eq:periodic-vf}
H^1_A([0,1], \R^n)=\Set{u \in H^1([0,1], \R^n)| u(0)=A u(1)}
\end{equation}
equipped with the inner product 
\begin{equation}\label{eq:3-5pigeon}
\langle \langle v,w\rangle\rangle _A\= \int_0^1 \big[\langle v'(s), A w'(s)\rangle+ \langle v(s), A w(s)\rangle\big]\, ds.
\end{equation}
Denoting by   $\Psi: \mathcal H(x) \to H^1_A([0,1], \R^n)$  the map defined by $\Psi(\xi)=u$ where $u(t)=E_t^{-1}(\xi(t))$, it follows that $\Psi$ is a linear isomorphism and it is easy to check that 
\begin{multline}
\xi(0)=	\xi(1)  \quad \iff\quad E_0 u(0)= E_1 u(1) \quad\iff\quad u(0)= A u(1) \textrm{ and } \\
\nabla_t \xi(0)= \nabla_t \xi(1)  \quad \iff\quad u'(0)= A u'(1)
\end{multline}
where, in the last, we used property of the frame being parallel.

For $i=1, \dots, n$ and $t \in[0,1]$, we let $e_i(t)\=E_t e_i$ and we denote by $\langle P(t) \cdot, \cdot \rangle$, $\langle Q(t) \cdot, \cdot \rangle$ and $\langle R(t) \cdot, \cdot \rangle$ respectively the pull-back by $E_t$ of $\rie{\bar  P(t) \cdot}{ \cdot }$, $\rie{ \bar Q(t) \cdot}{\cdot }$ and $\rie{ \bar R(t) \cdot}{\cdot }$. Thus, we get
\begin{multline}
P(t)\=[p_{ij}(t)]_{i,j=0}^n, \quad Q(t)\=[q^{ij}(t)]_{i,j=0}^n,\quad  R(t)\=[r_{ij}(t)]_{i,j=0}^n\quad  \textrm{ where }\\
p_{ij}(t)\=\rie{\bar{P}(t)e_i(t)}{e_j(t)}, \quad q_{ij}(t)\=\rie{\bar{Q}(t)e_i(t)}{e_j(t)},  \quad  r_{ij}(t)\	=\rie{ \bar{R}(t)e_i(t)}{e_j(t)}.
\end{multline}
We observe that $P$ and $ R$ are symmetric matrices and being  $e_i(T)=\sum_{j=1}^na_{ij}e_j(0)$ we get also that 
\begin{equation}\label{eq:condition-of-P-after-trivialization}
P(0)=AP(T)\trasp{A},\qquad P'(0)=AP'(T)\trasp{A}, \qquad Q(0)=AQ(T), \qquad R(0)=AR(T)\trasp{A}.
\end{equation}
Moreover, if $\partial_qL(t)=\sum_{i=1}^{n}l_i(t)e_i(t)$ and $x'(t)=\sum_{i=1}^{n}x_i(t)e_i(t)$, then we denote
\begin{equation}
L_q(t)\=\trasp{(l_1(t),l_2(t),\cdots,l_n(t))};\qquad  x'(t)\=\trasp{(x_1(t),x_2(t),\cdots,x_n(t))}.
\end{equation}

Now, for every $s \in [0,+\infty)$, the push-forward by $\Psi$ of the index forms $\mathcal I_s$ on $\mathcal H(x)\times \R$ is given by the symmetric bilinear forms  on $H^1_A([0,1], \R^n)\times \R$ defined by 
\begin{multline}\label{eq:path-index-2-free}
I_s[(u,b),(v,d)]\\ 
=\int_0^1 \left\{\big\langle\dfrac 1 T P(t)u'(t), v'(t)\big\rangle+ \langle Q(t)u(t), v'(t)\rangle+ \langle \trasp{Q}(t)u'(t), v(t)\rangle+ \langle TR(t)u(t), v(t)\rangle\right\} dt\\+\int_0^1\left\{  -\dfrac{1}{T^2}\langle P(t)x'(t), v'(t)\rangle\cdot b- \dfrac{1}{T^2}\langle  P(t)x'(t), u'(t)\rangle\cdot d +\langle L_q(t)-\dfrac 1 T Q(t)x'(t),u(t)\rangle\cdot d\right.\\+\left.\langle L_q(t)-\frac 1 T\trasp{Q}(t)x'(t),v(t)\rangle\cdot b+\dfrac{1}{T^3}\kappa(t)\cdot bd\right\}dt + s \alpha(x)[(\xi,b),(\eta,d)]\\
\textrm{ where } \alpha(x,T)[(u,b),(v,d)]\=\int_0^1 \left \{\langle \dfrac 1 T P(t)u(t), v(t) \rangle+\dfrac{1}{T^3}\kappa(t)bd\right\}\, dt.
\end{multline}
Denoting by $q^A_s$ the quadratic form on $H^1_A([0,1], \R^n)\times \R$ associate to $I_s$ then, as direct consequence of Proposition  \ref{thm:famiglia-Fredholm-free}, we get that
for every $s \in [0,+\infty)$, the quadratic form $q_s^A$ is Fredholm on $H^1_A([0,1], \R^n)\times \R$. 
The following result is crucial in the well-posedness of the spectral index. 
\begin{lem}\label{thm:non-degenerate-s_0-forme-free}
	Under the above notation, there exists $s_0 \in [0,+\infty)$ large enough such that for every $s \geq s_0$, the form 
	$I_s$ given in Equation \eqref{eq:path-index-2-free} is non-degenerate (in the sense of bilinear forms).
\end{lem}
\begin{proof}
	We argue by contradiction and we assume that for every $s_0 \ge 0$ there exists $s \geq s_0$ such that  $I_s$ is degenerate. Then there exists a $(u,b)\in H^1_A([0,1],\R^n)\times \R$ such that $I_s((u,b),(v,d))\equiv 0$ for every $(v,d)\in H^1_A([0,1],\R^n)\times \R$, namely we have
	\begin{multline}\label{eq:degenerate condition for index form free}
	0\equiv I_s((u,b),(v,d))\\
	=\int_0^1 \left\{\langle\dfrac 1 T P(t)u'(t), v'(t)\big\rangle+ \langle Q(t)u(t), v'(t)\rangle+ \langle \trasp{Q}(t)u'(t), v(t)\rangle+ \langle TR(t)u(t), v(t)\rangle\right\} dt\\+\int_0^1\left\{ -\dfrac{1}{T^2}\langle P(t)x'(t), v'(t)\rangle\cdot b- \dfrac{1}{T^2}\langle  P(t)x'(t), u'(t)\rangle\cdot d +\langle L_q(t)-\dfrac 1 T Q(t)x'(t),u(t)\rangle\cdot d \right.\\+\left.\langle L_q(t)-\dfrac 1 T\trasp{Q}(t)x'(t),v(t)\rangle\cdot b+\frac{1}{T^3}\kappa(t)\cdot bd\right\}dt\\
	+s\int_0^1 \left\{\langle \dfrac 1 T P(t)u(t), v(t) \rangle+\dfrac{1}{T^3}\kappa(t)bd\right\}\, dt.
	\end{multline}
We let $v(t)\=\frac 1 TP(t)u(t)$ and  we observe that as direct consequence of  Equation~\eqref{eq:condition-of-P-after-trivialization} it is admissible (meaning that $v$  belongs to $H^1_A$). Let $d\=\sign\kappa(t)\cdot b$, please keep in  mind that $(x,T)$ is non-null and consequently $\sign\kappa(t)\in\{1,-1\}$.  Now to be concise we omit $t$ in following expressions. By replacing $(v,d)$ into Equation~\eqref{eq:degenerate condition for index form free} with $v=\frac 1 TPu, d=\sign\kappa(t)\cdot b$,  we get 
	\begin{multline}\label{eq:inequality of non-degenerate condition for index form free}
	0\equiv I_s((u,b),(v,d)
	\\=\int_0^1 \{\big\langle\frac 1 T Pu', \frac 1 T(P'u+Pu')\big\rangle+ \langle Qu, \frac 1 T(P'u+Pu')\rangle+ \langle \trasp{Q}u', \frac 1 TPu\rangle+ \langle TRu, \frac 1 T P u\rangle\} dt\\+\int_0^1\big\{  -\frac{1}{T^2}\langle Px', \frac 1 T(P'u+Pu')\rangle\cdot b- \frac{1}{T^2}\langle  Px', u'\rangle\cdot \sign\kappa(t)\cdot b +\langle L_q-\frac 1 T Qx',u\rangle\cdot \sign\kappa(t)\cdot b \\+\langle L_q-\frac 1 T\trasp{Q}x',\frac 1 TPu\rangle\cdot b+\frac{1}{T^3}\kappa\cdot b\cdot \sign\kappa(t)\cdot b\big\}dt\\
	+s\int_0^1 \{\langle \frac 1 T Pu, \frac 1 T Pu \rangle+\frac{1}{T^3}\kappa b\cdot \sign\kappa(t)\cdot b\}\, dt\\
	\geq\int_{ 0 }^{ 1 }\Big[\frac{1}{T^2}\Vert Pu'\Vert^2-C_1\Vert Pu'\Vert \Vert Pu \Vert
	+(\frac{s}{T^2}-C_2)\Vert Pu\Vert^2\\
	-C_3\Vert Pu'\Vert \cdot b-C_4\Vert Pu\Vert\cdot  b+\frac{s+1}{T^3}\abs{\kappa}b^2 \Big]\, dt >0 
	\end{multline}
for $s$ large enough, where $C_i$ for $i=1,2,3,4$ are given as follows:
	\begin{multline}
	C_1=\frac{1}{T^2}\Vert P'P^{-1}\Vert+\frac 1 T\Vert QP^{-1}\Vert+\frac 1 T\Vert \trasp{Q}P^{-1}\Vert;\quad C_2=\frac 1 T\Vert QP^{-1}\Vert\Vert P'P^{-1}\Vert+\Vert RP^{-1}\Vert;\\
	C_3=\frac{1}{T^3}\Vert Px'\Vert+\frac{1}{T^2}\Vert Px'\Vert\Vert P^{-1}\Vert+\frac 1 T\Vert L_q-\frac 1 T\trasp{Q}x'\Vert;\\ C_4=\frac{1}{T^3}\Vert Px'\Vert\Vert PP^{-1}\Vert+\Vert L_q-\frac 1 TQx'\Vert\Vert P^{-1}\Vert.
	\end{multline}
By this contradiction we conclude the proof.
\end{proof}
Now, for every $s \in [0,+\infty)$, the push-forward by $\Psi$ of the index forms $\mathcal I_s$ on $\mathcal H(x)$ is given by the symmetric bilinear forms  on $H^1_A([0,1], \R^n)$ defined by 
\begin{multline}\label{eq:path-index-2-fix}
I_{s,T}[u,v]\= 
\int_0^1 \{\big\langle\frac 1 T P(t)u'(t), v'(t)\big\rangle+ \langle Q(t)u(t), v'(t)\rangle+ \langle \trasp{Q}(t)u'(t), v(t)\rangle+ \langle TR(t)u(t), v(t)\rangle\} dt\\
\textrm{ where } \alpha_T(x)[u,v]\=\int_0^1 \langle \frac 1 T P(t)u(t), v(t) \rangle\, dt.
\end{multline}
Similarly, we can prove that there exists a $s_0>0$ large enough such that $I_{s,T}$ is non-degenerate for every $s>s_0$. Therefore, we have following result for both free period and fixed period cases.
\begin{prop}\label{thm:spectral-index-well-defined}
	Let $(x,T)$ be a non-null critical point of Lagrangian action  \eqref{eq:free period functional}, then the spectral indices  introduced  in Definition~\ref{def:spectral-index} are well-defined.
\end{prop}
\begin{proof}
	We start observing  that,  $s\mapsto q^A_s$ is a path of Fredholm quadratic forms on $H^1_A([0,T], \R^n)\times \R$. Moreover, by Lemma \ref{thm:non-degenerate-s_0-forme-free}, there exists $s_0 \in [0,+\infty)$ such that $q^A_s$ is non-degenerate for every $s \geq s_0$ and hence the integer $\spfl\{q^A_s, s \in [0,s_0]\}$ is well-defined.  
	
	The conclusion follows by observing that $q^A_s$ is the push-forward by $\Psi$  of the Fredholm quadratic form  $\mathcal Q_s$ and by the fact that the spectral flow of a generalized family of Fredholm quadratic forms on the (trivial) Hilbert bundle $[0,s_0] \times (\mathcal H(x)\times \R)$ is independent on the trivialization. This concludes the proof. 	
\end{proof}

\subsection{The difference between two spectral indices}
Since both spectral indices defined in Definition \ref{def:spectral-index} are well-defined, a natural question arises that what relationship is between them. In this subsection we will give the answer. Firstly we will give an abstract formula to compute the difference between two spectral flows. Reader is referred to \cite[Appendix C]{KF07}, but we will drop the regularity assumption in  \cite[Appendix C, Theorem C.5.]{KF07}.

Let $H$ be a real separable Hilbert space with inner product $\langle\cdot,\cdot\rangle$, $W\subset H$ is dense and the inclusion map  $i:W\rightarrow H$ is compact. Let $A$ be an  unbounded linear operator on $H$ with $W$ as its domain. Here we assume  that $A$ is a Fredholm  operator. For a given finite dimensional Hilbert space $V$, assume $B:V\rightarrow H$ is a bounded linear operator and  $C:V\rightarrow V$ is a bounded self-adjoint linear operator. Denote by $\mathcal A:W\oplus V\rightarrow H\oplus V$ the self-adjoint operator defined by 
\begin{equation}
\mathcal A(w,v)=(Aw+Bv,B^*w+Cv),
\end{equation}
where $B^*$ is the adjoint operator of $B$. For convenience, we will express $\mathcal A$ in matrix form, namely $\mathcal A=\begin{bmatrix}
A&B\\B^*&C
\end{bmatrix}$. 

\begin{lem}\label{lem:C0spfl}
	Under above assumptions, we have \begin{equation}\label{eq:moe-lem}
	m^-\Big(\begin{bmatrix}
	0&B\\B^*&C
	\end{bmatrix}\Big)=m^-(C|_{\ker B})+\dim (\ima B),
	\end{equation}
	where $m^-$ denotes the Morse index.
\end{lem}
\begin{proof}
	For $V$ we have following splitting:
	\begin{equation}
V=\ima B^*\oplus \ker B=\ima B\oplus\ker B^*.
	\end{equation}
By choosing suitable basis, the quadratic form  $\begin{bmatrix}
0&B\\B^*&C
\end{bmatrix}$ has block form
$\begin{bmatrix}
	0&0&B_{11}&0\\0&0&0&0\\B_{11}^*&0&C_{11}&C_{12}\\0&0&C^*_{12}&C_{22}
\end{bmatrix}$, where $B_{11}:\ima B^*\rightarrow \ima B$ and $B^*_{11}:\ima B\rightarrow \ima B^*$ are both invertible. Therefore, we can get that $\begin{bmatrix}
0&B\\B^*&C
\end{bmatrix}$ is similar to $\begin{bmatrix}
0&0&B_{11}&0\\0&0&0&0\\B_{11}^*&0&0&0\\0&0&0&C_{22}
\end{bmatrix}$. Note that $m^-\Big (\begin{bmatrix}
0&B_{11}\\B_{11}^*&0\end{bmatrix}\Big )=\dim(\ima B)$ and $m^-(C|_{\ker B})=m^-(C_{22}|_{\ker B})$. By this arguments the thesis follows.  
\end{proof}

\begin{lem}\label{lem:abstract-As-spfl}
	We let $\mathcal A(s)=\begin{bmatrix}
	A&(1-s)B\\(1-s)B^*&(1-s)C
	\end{bmatrix}$ for  $s\in[0,1]$. Then, the following equality holds
	\begin{equation}
	\spfl(\mathcal A(s), s\in[0,1])=m^-(\mathcal A(0)\vert_{W^{\bot}})+\dim(W\cap W^{\bot})-\dim(W\cap \ker \mathcal A(0)).
	\end{equation}
\end{lem}
\begin{proof}
Consider the splitting $W=(\ker A)^\bot\oplus\ker A$, then 	$\mathcal A(s)$ is in the following block form
\begin{equation}
\mathcal A(s)=\begin{bmatrix}
A_{11}&0&(1-s)B_1\\0&0&(1-s)B_2\\(1-s)B_1^*&(1-s)B_2^*&(1-s)C
\end{bmatrix},
\end{equation} 
 where  $A_{11}:(\ker A)^\bot\rightarrow (\ker A)^\bot$ is invertible and $B_1: V\rightarrow (\ker A)^\bot, B_2:V\rightarrow \ker A$.
 
 Let $\mathcal B(s)=\begin{bmatrix}
 A_{11}&0&0\\0&0&(1-s)B_2\\0&(1-s)B_2^*&(1-s)[C-(1-s)B_1^*A_{11}^{-1}B_{1}]
 \end{bmatrix},s\in[0,1]$, then we have 
 \begin{equation}\label{eq:spflAs=Bs}
 \spfl\{\mathcal A(s), s\in[0,1]\}=\spfl\{\mathcal B(s),s\in[0,1]\}.
 \end{equation}

We now  prove the equality provided in Equation~$\eqref{eq:spflAs=Bs}$ by using the  homotopy invariant property of the  spectral flow.
So, let's start to consider the $2$-parameter family of operators pointwise defined by 
	\[
	A(s,t)=\begin{bmatrix}
		A_{11}&0& t(1-s)B_1\\0&0&(1-s)B_2\\t(1-s)B_1^*&(1-s)B_2^*&(1-s)[C-(1-t)^2(1-s)B_1^*A_{11}^{-1}B_{1}]
	\end{bmatrix}
	\] 
	In fact we have $A(s,t)=K(t)^* A(s) K(t) $ with 
	\[
	K(t)=\begin{bmatrix}
		\Id &0& -(1-t)(1-s)A_{11}^{-1} B_1\\0 &\Id &0\\ 0&0& \Id 
	\end{bmatrix}.
	\]
By a straightforward calculation it follows that  $\dim \ker A(0,t)$ and $\dim \ker A(1,t)$ are both constants. By using the zero axiom of the spectral flow (namely each path is contained in a fixed stratum of the Fredholm Lagrangian Grassmannian), we get that   $\spfl\{A(0,t),t\in [0,1]\}=\spfl\{A(1,t),t\in [0,1]\}=0$.
	Therefore, by using the homotopy invariant property, we get that  $\spfl\{A(s,0),s\in [0,1]\}=\spfl\{A(s,1)s\in [0,1]\}$ which is precisely Equation~\eqref{eq:spflAs=Bs}.

Let $\mathcal C(s)=\begin{bmatrix}
0&(1-s)B_2\\(1-s)B_2^*&(1-s)[C-(1-s)B_1^*A_{11}^{-1}B_{1}]
\end{bmatrix},s\in[0,1]$. By invoking the additivity property of spectral flow under direct sum, we get $\spfl\{ \mathcal B(s),s\in[0,1]\}=\spfl\{\mathcal C(s),s\in[0,1]  \}$. By \eqref{eq:spflAs=Bs} we only need to compute $\spfl\{\mathcal C(s),s\in[0,1]  \}$. 

By using Lemma~\ref{lem:C0spfl} we compute such a spectral flow. Since $\mathcal C(1)=0$,	then we have 
	\begin{equation}\label{eq:spfl=morC0}
		\spfl\{\mathcal C(s),s\in[0,1] \}=m^-(C(0)). 
	\end{equation}
	
	By Lemma~\ref{lem:C0spfl}	we have
	\begin{equation}\label{eq:morseC0}
		\spfl\{\mathcal C(s),s\in[0,1] \}=m^-(C(0))=m^-((C-B_1^*A_{11}^{-1}B_1)|_{\ker B_2})+\dim(\ima B_2).
	\end{equation} 
	The next step is to prove  
	\begin{align}
		&m^{-}(\mathcal A(0)|_{W^{\bot}})=m^-((C-B_1^*A_{11}^{-1}B_1)|_{\ker B_2})\label{eq:reduce_index}\quad \textrm{ and }\\
		&\dim(W\cap W^{\bot})-\dim(W\cap\ker \mathcal{A}(0))=\dim(\ima B_2)\label{eq:reduce_kernel}.
	\end{align}
Let us consider $\trasp{(x_1,x_2,0)}\in \ker \mathcal A(0)\cap W$. Then for every $\trasp{(u_1,u_2,v)}\in W\oplus V$ we have 
\begin{equation}\label{eq:condi-kerA}
\begin{aligned}
\Big\langle \begin{bmatrix}
A_{11}&0&B_1\\0&0&B_2\\B_1^*&B_2^*&C
\end{bmatrix}\begin{bmatrix}
x_1\\x_2\\0
\end{bmatrix},\begin{bmatrix}
u_1\\u_2\\v
\end{bmatrix}\Big\rangle=\langle A_{11}x_1,u_1 \rangle+\langle B_1^*x_1,v \rangle+\langle B_2^*x_2,v \rangle\equiv 0.
\end{aligned}
\end{equation}
We set $v=0$. So, $\langle A_{11}x_1,u_1 \rangle\equiv0$ for every $u_1\in(\ker A)^{\bot}$ implies that  $A_{11}x_1=0$. Consequently we have $x_1=0$. Now equation \eqref{eq:condi-kerA} becomes $\langle B_2^*x_2,v \rangle\equiv 0$. Since $v$ is arbitrary, then $B_2^*x_2=0$. Hence $W\cap \ker \mathcal A(0)=\{ \trasp{(0,x_2,0)}\ | \ B_2^*x_2=0\}=\ker B_2^*$.

If $\trasp{(x_1,x_2,y)}\in W^{\bot}$, then for every $\trasp{(u_1,u_2,0)}\in W$ we have 
\begin{equation}\label{eq:condi-W-ortho}
\begin{aligned}
\Big\langle \begin{bmatrix}
A_{11}&0&B_1\\0&0&B_2\\B_1^*&B_2^*&C
\end{bmatrix}\begin{bmatrix}
x_1\\x_2\\y
\end{bmatrix},\begin{bmatrix}
u_1\\u_2\\0
\end{bmatrix}\Big\rangle=\langle A_{11}x_1,u_1 \rangle+\langle B_1y,u_1\rangle+\langle B_2 y,u_2\rangle\equiv 0.
\end{aligned}
\end{equation}
We set  $u_2=0$. Then $\langle A_{11}x_1+B_1y,u_1 \rangle\equiv0$ for every $u_1\in(\ker A)^{\bot}$ implies that $A_{11}x_1+B_1y=0$. Consequently we have $x_1=-A_{11}^{-1}B_1y$. Let $u_1=0$, equation \eqref{eq:condi-W-ortho} becomes $\langle B_2y,u_2 \rangle\equiv 0$ for every $u_2\in\ker A$, then $B_2y=0$. Hence $W^{\bot}=\{ \trasp{(-A_{11}^{-1}B_1y,x_2,y)}\ | \ B_2y=0, x_2\in\ker A\}$ and $W\cap W^{\bot}=\{ \trasp{(0,x_2,0)}\}=\ker A$.

Now, we get 
\[
\dim W\cap W^{\perp}- \dim W\cap \ker \mathcal A(0)=\dim \ker A-\dim \ker B_2^*=\dim \ima B_2.
\]

For every $\xi_0=\trasp{(-A_{11}^{-1}B_1y,x_2,y)}\in W^{\bot}$, we have 
	\begin{equation}
	\begin{aligned}
	\langle \mathcal A(0)\xi_0,\xi_0\rangle&=\Big\langle \begin{bmatrix}
	A_{11}&0&B_1\\0&0&B_2\\B_1^*&B_2^*&C
	\end{bmatrix}\begin{bmatrix}
	-A_{11}^{-1}B_1y\\x_2\\y
	\end{bmatrix},\begin{bmatrix}
	-A_{11}^{-1}B_1y\\x_2\\y
	\end{bmatrix}\Big\rangle\\
	&=\langle-B_1y,-A_{11}^{-1}B_1y\rangle+\langle B_1y,-A_{11}^{-1}B_1y\rangle+\langle B_2y,x_2\rangle\\
	&\qquad+\langle -B_1^*A_{11}^{-1}B_1y,y\rangle+\langle B_2^*x_2,y\rangle+\langle Cy,y\rangle\\
	&=\langle (C-B_1^*A_{11}^{-1}B_1)y,y\rangle.
	\end{aligned}
	\end{equation}
Therefore, we have $m^{-}(\mathcal A(0)|_{W^{\bot}})=m^-((C-B_1^*A_{11}^{-1}B_1)|_{\ker B_2})$. This concludes the proof.
\end{proof}

\begin{lem}\label{lem:abstract-As0-spfl}
	We let $\mathcal A(s)=\begin{bmatrix}
	A&sB\\sB^*&sC
	\end{bmatrix}$, for $s\in[0,1]$ and we assume that $A$ is invertible. Then we have 
	\begin{equation}
	\spfl(\mathcal A(s), s\in[0,1])=-m^-(C-B^*A^{-1}B).
	\end{equation}
\end{lem}
\begin{proof} 
By a similar discussion as provided in the proof of Equation \eqref{eq:spflAs=Bs}, we get
	\begin{eqnarray}
		\spfl(\mathcal A(s), s\in[0,1])=\spfl(\mathcal B(s),s\in[0,1]  ),
		\ \text{where}   \ \mathcal B(s)=\begin{bmatrix}
			A&0\\0&s[C-sB^*A^{-1}B]
		\end{bmatrix},
	\end{eqnarray}
	then, by the invertibility of $A$, we infer
	\begin{equation}
		\spfl(\mathcal A(s), s\in[0,1])=\spfl(s(C-sB^*A^{-1}B) ,s\in[0,1]  ).
	\end{equation}
	We observe that the operator $s(C-sB^*A^{-1}B)$ is defined on a finite dimensional space $V$, and  by invoking Proposition \ref{thm:sf-differenza-morse}, we  conclude that  
	\begin{equation}
	\spfl(s(C-sB^*A^{-1}B) ,s\in[0,1] ) =-m^-(C-B^*A^{-1}B).
	\end{equation}	
\end{proof}
In Lemma \ref{thm:non-degenerate-s_0-forme-free} we proved that $A(s_0)$ is invertible. The next result provided a more strinking propery about the spectrum of $A(s_0)$.
\begin{lem}\label{lem:spectrum-As0}
There exists $\delta>0$ such that $\sigma(A(s_0))\cap[-\delta,\delta]=\emptyset$. In particular, the operator  $A^{-1}(s_0)$	 is bounded.
\end{lem} 
\begin{proof}
Arguing by contradiction, we assume that for every $\delta>0$ there exists $\lambda_\delta\in[-\delta,\delta]$ and $u_\delta\in H^1_A([0,1], \R^n)$ such that $A(s_0)\trasp{(u'_{\delta},u_\delta)}=\lambda_\delta \trasp{(u'_{\delta},u_\delta)}$. Take $v_\delta=Pu_\delta$, then we have 
\begin{equation}\label{eq:inequality-less}
\begin{aligned}
I_{s_0,T}(u_\delta,v_{\delta})&=\Big\langle A(s_0)\begin{bmatrix}
u_\delta'\\u_\delta
\end{bmatrix},\begin{bmatrix}
v'_\delta\\v_\delta
\end{bmatrix}\Big\rangle=\lambda_\delta\Big\langle \begin{bmatrix}
u_\delta'\\u_\delta
\end{bmatrix},\begin{bmatrix}
v'_\delta\\v_\delta
\end{bmatrix}\Big\rangle\\
&=\lambda_\delta\int_0^1[\langle \ud',P '\ud\rangle+\langle\ud',P\ud ' \rangle+\langle \ud,P\ud\rangle]dt\\
&\leq \abs{\lambda_\delta}\int_0^1[\Vert P^{-1}\Vert\Vert P\ud'\Vert^2+\Vert P^{-1}\Vert\Vert P'P^{-1}\Vert\Vert P\ud'\Vert\Vert P\ud\Vert+\Vert P^{-1}\Vert\Vert P\ud\Vert^2]dt.
\end{aligned}
\end{equation}
But by \eqref{eq:inequality of non-degenerate condition for index form free} we have 
\begin{equation}\label{eq:inequality-grate}
\begin{aligned}
 I_{s_0,T}(\ud,v_\delta)
\geq\int_{ 0 }^{ 1 }\Big[\frac{1}{T^2}\Vert P\ud'\Vert^2-C_1\Vert P\ud'\Vert \Vert P\ud \Vert
+(\frac{s_0}{T^2}-C_2)\Vert P\ud\Vert^2\Big]\, dt.
\end{aligned}
\end{equation}
Inequalities \eqref{eq:inequality-less} and \eqref{eq:inequality-grate} are contradiction for $s_0$ large enough and $\delta$ (consequently $\lambda_\delta$) small enough. This concludes the proof.
\end{proof}
Now we are ready to compute the difference between two spectral indices defined in Definition \ref{def:spectral-index}. By taking into account  Equation~\eqref{eq:path-index-2-free}, we get
\begin{equation}
I_s[(u,b),(v,d)]=\Big\langle \mathcal A(s)\begin{bmatrix}
u'\\u\\b
\end{bmatrix},\begin{bmatrix}
v'\\v\\d
\end{bmatrix}\Big\rangle,
\end{equation}
where $\mathcal A(s)=\begin{bmatrix}
\frac 1 T P&Q&-\frac{1}{T^2}Px'\\\trasp{Q}&s\frac 1 TP+TR&L_q-\frac 1 T\trasp{Q}x'\\-\frac{1}{T^2}\trasp{(Px')}&\trasp{(L_q-\frac 1 T \trasp{Q}x')}&(s+1)\frac{1}{T^3}\kappa
\end{bmatrix}.$ 

We set
\begin{equation}
A(s)=\begin{bmatrix}
\frac 1 T P&Q\\\trasp{Q}&s\frac 1 TP+TR
\end{bmatrix}, \quad B=\begin{bmatrix}
-\frac{1}{T^2}Px'\\L_q-\frac 1 T\trasp{Q}x'
\end{bmatrix}\quad \textrm{ and }  C(s)=(s+1)\frac{1}{T^3}\kappa
\end{equation}
and we  consider the homotopy $\mathcal A(s,\epsilon)=\begin{bmatrix}
A(s)&(1-\epsilon)B\\(1-\epsilon)\trasp{B}&(1-\epsilon)C(s)
\end{bmatrix}$. By the homotopy invariance property of spectral flow we get
\begin{equation}\label{eq:homotopy-spfl-sum}
\begin{aligned}
\spfl (\mathcal A(s,0), s\in[0,s_0] ) =\spfl (\mathcal A(0,\epsilon), \epsilon \in [0,1] )  &+\spfl(\mathcal A(s,1) , s\in[0,s_0] )\\& +\spfl(\mathcal A(s_0,1-\epsilon),\epsilon\in[0,1]).
\end{aligned}
\end{equation}
Let us now compute the spectral flow $\spfl(\mathcal A(s_0,1-\epsilon),\epsilon\in[0,1])$.
By Lemma \ref{lem:spectrum-As0} we infer that $B^*A^{-1}(s_0)B$ is a bounded operator (on a one-dimensional space). So, we get 
\begin{equation}
m^-(C(s_0)-B^*A^{-1}(s_0)B)=\left\{
\begin{aligned}
&1\quad \text{if}\quad \kappa<0\\
&0\quad \text{if}\quad \kappa>0.
\end{aligned}
\right.
\end{equation}
By Lemma \ref{lem:abstract-As0-spfl}, we have 
\begin{equation}\label{eq:spfl-s0-epsilon}
\spfl(\mathcal A(s_0,1-\epsilon),\epsilon\in[0,1]) =m^-(C(s_0)-B^*A^{-1}(s_0)B)=\left\{
\begin{aligned}
&1\quad \text{if}\quad \kappa<0\\
&0\quad \text{if}\quad \kappa>0.
\end{aligned}
\right.
\end{equation}
Let us now compute the spectral flow $\spfl (\mathcal A(0,\epsilon), \epsilon \in [0,1] ) $. By Lemma \ref{lem:abstract-As-spfl} we have 
\begin{equation}
\spfl(\mathcal A(0,\epsilon), \epsilon \in [0,1]) =m^-(\mathcal A(0,0)\vert_{W^{\bot}})+\dim(W\cap W^{\bot})-\dim(W\cap \ker \mathcal A(0,0))
\end{equation}
where $W=H^1_A([0,1], \R^n)$ and $V=\R$. 

The next step is to provide an explicit description of 
\[
m^-(\mathcal A(0,0)\vert_{W^{\bot}})+\dim(W\cap W^{\bot})-\dim(W\cap \ker \mathcal A(0,0)).
\] 
The basic idea comes from \cite[Section 2.1]{MP10}.

Let $(x,T)$ be a non-null critical point of $\eh$ with orbit cylinder $(x_{h+s},T_{h+s})$. Then  for every $(\xi,b)\in\mathcal H(x)\times \R $ we have 
\begin{equation}\label{eq:cylider-crirtical}
d\mathbb{E}_{h+s}(x_{h+s},T_{h+s})[(\xi,b)]\equiv0.
\end{equation}
By differentiating w.r.t. $s$  both sides  of Equation~\eqref{eq:cylider-crirtical}, we get
\begin{equation}\label{eq:cylider-derivative}
\begin{aligned}
d^2\eh(x,T)[(\xi_h,T'(h)),(\xi,b)]+\frac{\partial}{\partial s}\Big|_{s=0}d\mathbb{E}_{h+s}(x,T)[(\xi,b)]=0,
\end{aligned}
\end{equation}
where $\xi_h(t)=\frac{\partial}{\partial s}|_{s=0}x_{h+s}(t),T'(h)=\frac{d}{ds}|_{s=0}T_{h+s}.$  Let now choose a variation $\{(x_{h,r},T_{h,r}),r\in(-\epsilon,\epsilon)\}$  such that $(x_{h,0},T_{h,0})=(x,T)$ and $\frac{\partial}{\partial r}|_{r=0}(x_{h,r},T_{h,r})=(\xi,b)$, then we have 
\begin{equation}
\begin{aligned}
\frac{\partial}{\partial s}\Big|_{s=0}d\mathbb{E}_{h+s}(x,T)[(\xi,b)]&=\frac{\partial}{\partial s}\Big|_{s=0}\frac{\partial}{\partial r}\Big|_{r=0}\mathbb{E}_{h+s}(x,T)[(x_{h,r},T_{h,r})]\\
&=\frac{\partial}{\partial r}\Big|_{r=0}\frac{\partial}{\partial s}\Big|_{s=0}\mathbb{E}_{h+s}(x,T)[(x_{h,r},T_{h,r})]\\
&=\frac{\partial}{\partial r}\Big|_{r=0}T_{h,r}=b.
\end{aligned}
\end{equation}
By taking into account Equation~\eqref{eq:cylider-derivative} we have 
\begin{equation}\label{eq:not-in-ker}
d^2\eh(x,T)[(\xi_h,T'(h)),(\xi,b)]=-b
\end{equation}
for every $(\xi,b)\in \mathcal H(x)\times \R$. Taking $b=0$ and $(\xi,b)=(\xi_h,T'(h))$ respectively, we have 
\begin{equation}\label{eq:one-more-vector}
d^2\eh(x,T)[(\xi_h,T'(h)),(\xi,0)]=0,\quad d^2\eh(x,T)[(\xi_h,T'(h)),(\xi_h,T'(h))]=-T'(h).
\end{equation}
Let us identify $\mathcal H(x)$ with $\mathcal H(x)\times \{0\}$ and  we denote the Hessian of $\eh$ by $\nabla^2\eh(x,T)$. So, $\ker d^2\eh(x,T)=\ker \nabla^2\eh(x,T)$. We now set 
\begin{equation}
\mathcal H^\bot(x)=\{(\xi,b)\in\mathcal H(x)\times \R \ | \ d^2\eh(x,T)[(\xi,b),(\eta,0)]=0,\ \forall (\eta,0)\in\mathcal H(x)\}.
\end{equation} 
The following result holds. 
\begin{lem}\label{lem:nondeg-cylinder}
	Under above notations, we get 
	\begin{equation}
\ker d^2\eh(x,T)\subset \mathcal H(x),\quad \textrm{ and } \quad 	\mathcal H^\bot(x)=\ker d^2\eh(x,T)\oplus\R(\xi_h,T'(h)).
	\end{equation}
\end{lem}
\begin{proof}
	We argue by contradiction. If $\mathcal H(x)+\ker d^2\eh(x,T)=\mathcal H(x)\times \R$, then 
	\begin{equation}\label{eq:H-ortho=ker}
			\mathcal H^\bot(x)=(\mathcal H(x)+\ker d^2\eh(x,T))^\bot=(\mathcal H(x)\times \R)^\bot=\ker d^2\eh(x,T).
	\end{equation}
By taking into account Equation~\eqref{eq:one-more-vector}, we get  $(\xi_h,T'(h))\in\mathcal H^\bot(x)$ and by Equation~\eqref{eq:not-in-ker} we know  $(\xi_h,T'(h))\notin\ker d^2\eh(x,T)$ which is in contrast with \eqref{eq:H-ortho=ker}. \\ Therefore, $\mathcal H(x)+\ker d^2\eh(x,T)\neq\mathcal H(x)\times \R$. Since $\dim((\mathcal H(x)\times \R)/\mathcal H(x))=1,$ then  we get $\ker d^2\eh(x,T)\subset \mathcal H(x)$. 

We now observe that $\ker d^2\eh(x,T)\oplus\R(\xi_h,T'(h))\subset \mathcal H^\bot(x)$ and $\dim(\mathcal H^\bot(x)/\ker d^2\eh(x,T))\leq1$, it follows that $\mathcal H^\bot(x)=\ker d^2\eh(x,T)\oplus\R(\xi_h,T'(h))$. This concludes the proof.	
\end{proof}
By invoking Lemma~\ref{lem:nondeg-cylinder}, we get 
	\[
\mathcal H^\bot(x)=	\ker d^2\eh(x,T)\oplus \R(\xi_h,T'(h))\supset \mathcal H(x)\cap \mathcal H^\bot(x)\supset \ker d^2\eh(x,T).
	\] 
	If $T'(h)\neq0$, the $\R(\xi_h,T'(h))\subsetneq \mathcal H(x)$.
Thus, we have $\mathcal H(x)\cap\mathcal H^\bot(x)=\ker d^2\eh(x,T)$. 

At the same time, by Lemma \ref{lem:nondeg-cylinder} we have $d^2\eh(x,T)|_{\mathcal H^\bot(x)}=d^2\eh(x,T)|_{\R(\xi_h,T'(h))}$. Then by Equation~\eqref{eq:one-more-vector}, we have 
\begin{equation}
m^-(d^2\eh(x,T)|_{\mathcal H^\bot(x)})=
\left\{
\begin{aligned}
&1\qquad \text{if}\quad T'(h)>0\\
&0\qquad \text{if}\quad T'(h)<0.
\end{aligned}
\right.
\end{equation}
Hence, we have 
\begin{equation}
\begin{aligned}
m^-(d^2\eh(x,T)|_{\mathcal H^\bot(x)})+\dim(\mathcal H(x)\cap\mathcal H^\bot(x))&-\dim(\mathcal H(x)\cap \ker d^2\eh(x,T))\\&=
\left\{
\begin{aligned}
&1\quad \text{if}\quad T'(h)>0\\
&0\quad \text{if}\quad T'(h)<0.
\end{aligned}
\right.
\end{aligned}
\end{equation}
If $T'(h)=0$, then by Equation~\eqref{eq:one-more-vector}, we get that $(\xi_h,T'(h))\in \mathcal H(x)\cap\mathcal H^\bot(x)$. Therefore, we have 
 \[
 \mathcal H(x)\cap\mathcal H^\bot(x)=\ker d^2\eh(x,T)\oplus \R(\xi_h,0)=\mathcal H^\bot (x).\] 
 
 As a result, we have $\dim(\mathcal H(x)\cap\mathcal H^\bot(x))-\dim(\mathcal H(x)\cap \ker d^2\eh(x,T))=1$,  and  
 \[
 d^2\eh(x,T)|_{\mathcal H^\bot(x)}=d^2\eh(x,T)|_{\mathcal H^\bot(x)\cap \mathcal H(x)} =0.
 \] 
 So, we get that if $T'(h)=0$ there holds 
\begin{equation}
m^-(d^2\eh(x,T)|_{\mathcal H^\bot(x)})+\dim(\mathcal H(x)\cap\mathcal H^\bot(x))-\dim(\mathcal H(x)\cap \ker d^2\eh(x,T))=1.
\end{equation}  

Summing up, we have 
\begin{equation}
\begin{aligned}
m^-(d^2\eh(x,T)|_{\mathcal H^\bot(x)})+\dim(\mathcal H(x)\cap\mathcal H^\bot(x))&-\dim(\mathcal H(x)\cap \ker d^2\eh(x,T))\\&=
\left\{
\begin{aligned}
&1\quad \text{if}\quad T'(h)\ge 0\\
&0\quad \text{if}\quad T'(h)<0.
\end{aligned}
\right.
\end{aligned}
\end{equation}
So, in conclusion we get 
\begin{equation}\label{eq:spfl-0-epsilon}
\begin{aligned}
\spfl(\mathcal A(0,\epsilon), \epsilon \in [0,1] ) &=m^-(\mathcal A(0,0)\vert_{W^{\bot}})+\dim(W\cap W^{\bot})-\dim(W\cap \ker \mathcal A(0,0))\\&=
m^-(d^2\eh(x,T)|_{\mathcal H^\bot(x)})+\dim(\mathcal H(x)\cap\mathcal H^\bot(x))\\& -\dim(\mathcal H(x)\cap \ker d^2\eh(x,T))\\&=
\left\{
\begin{aligned}
&1\quad \text{if}\quad T'(h)\ge 0\\
&0\quad \text{if}\quad T'(h)<0
\end{aligned}
\right.,
\end{aligned}
\end{equation}
where $W=H^1_A([0,1], \R^n)$. Summing up all computations, we can give the precise value of the difference between two spectral indices defined in Definition \ref{def:spectral-index}.
\begin{thm}
Under above notations the following equalities hold:
\begin{equation}\label{eq:difference-spfl}
\begin{aligned}
\ispec(x,T)-\ispec^T(x)&=\spfl\big(\mathcal Q_s, s \in [0, s_0]\big)-\spfl\big(\mathcal Q_{s,T}, s \in [0, s_0]\big)\\
&=\spfl (\mathcal A(s,0), s\in[0,s_0] )-\spfl(\mathcal A(s,1) , s\in[0,s_0] )\\
&=\spfl (\mathcal A(0,\epsilon), \epsilon \in [0,1] )  +\spfl(\mathcal A(s_0,1-\epsilon),\epsilon\in[0,1])\\
&=\left\{
\begin{aligned}
&2\quad \text{if}\quad \kappa<0,\ T'(h)\ge 0\\
&1\quad \text{if}\quad \kappa<0, \  T'(h)< 0\quad \text{or}\quad \kappa>0, \  T'(h)\ge 0\\
&0\quad \text{if}\quad \kappa>0,\ T'(h)< 0\\
\end{aligned}
\right.
\end{aligned}
\end{equation}	
\end{thm}
\begin{proof}
	The proof readily follows by invoking Equations \eqref{eq:homotopy-spfl-sum}-\eqref{eq:spfl-s0-epsilon} and Equation~\eqref{eq:spfl-0-epsilon}. This concludes the proof.
\end{proof}
\begin{rem}
We observe that the main role of orbit cylinder is to ensure the existence of vector $(\xi_h,T'(h))$ in Equation~\eqref{eq:not-in-ker}. A natural problem is to find out some more general conditions to insure the existence of a vector in $\mathcal H^\bot(x)$ but not in $\ker d^2\eh(x,T)$. The bifurcation theory of Hamiltonian system could be the  right direction for answering  this question. 
\end{rem} 
By using  Proposition \ref{thm:spectral-index-morse-index},  the following result holds. 
\begin{cor}\label{cor:morse-difference}
If $L$ is $\mathscr C^2$-strictly convex on $TM$, then the difference between two Morse indices is given by 
\begin{equation}
\iMorse(x,T)-\iMorse^T(x)=
\left\{
\begin{aligned}
&1\qquad \text{if}\quad T'(h)\ge 0\\
&0\qquad \text{if}\quad T'(h)<0
\end{aligned}
\right.
\end{equation}
\end{cor}
\begin{proof}
	Since $L$ is $\mathscr C^2$-strictly convex, then $\kappa=\langle Px',x'\rangle>0$.  By Proposition \ref{thm:spectral-index-morse-index} and Equation~\eqref{eq:difference-spfl}, we conclude the proof.
\end{proof}
\begin{rem}
	Here we point out that if $(x,T)$ is a minimizer of system \eqref{eq:fixed period functional} when $L$ is $\mathscr C^2$-strictly convex, then we must have $\iMorse(x,T)=0$. By Corollary \ref{cor:morse-difference} there must hold $T'(h)<0$. Or vice versa, if $T'(h)\ge 0$, then $(x,T)$ cannot be a minimizer.
\end{rem}

%%%%%%%%%%%%%%%%%%%%%%%%%%%%%%%%%%%%%%%%%%%%%%%%%%%%%%%%%%%%
%%%%%%%
%%%%%%%
%%%%%%%
%%%%%%%%%%%%%%%%%%%%%%%%%%%%%%%%%%%%%%%%%%%%%%%%%%%%%%%%%%%%

\section{Linear instability and proof of the main result}\label{sec:linear instability}

In this section we recall some well-known results about  the fixed period problem. We refer the interested reader to \cite{PWY21} for the complete proofs.

%%%%%%%%%%%%%%%%%%%%%%%%%%%%%%%%%%%%%%%%%%%%%%%%%%%%%%%%%%%%
%%%%%%%
%%%%%%%
%%%%%%%
%%%%%%%%%%%%%%%%%%%%%%%%%%%%%%%%%%%%%%%%%%%%%%%%%%%%%%%%%%%%

\subsection{Hamiltonian system and geometrical index}

It is well-known that  under the assumption (N1) the Legendre transform 
\begin{equation}
\mathscr L_L:TM \to  T^*M, \qquad (q,v)\mapsto\Big(q,DL(q,v)\big\vert_{T^v_{(q,v)}TM}\Big)
\end{equation}
is a local smooth diffeomorphism.  The Fenchel transform of $L$ is the autonomous Hamiltonian on $T^*M$ 
\begin{equation}
H(q,p)=\max_{v \in T_qM}\big(p[v]-L(q,v)\big)=p[v(q,p)]- L(q,v(q,p)),	
\end{equation}
for every $(q,p) \in T^*M$, where the map $v$ is a component of the fiber-preserving diffeomorphism 
\[
\mathscr L_L^{-1}: T^*M \to TM, \qquad (q,p)\mapsto (q,v(q,p))
\]
the inverse of $\mathscr L_L$.

By the above Legendre transform, the Euler-Lagrange Equation \eqref{eq:E-L equation-fixed} is changed into the following Hamiltonian equation:
\begin{equation}\label{eq:hamiltonian-manifold}
z_x'(t)=J\nabla H(z_x(t)).
\end{equation}

By trivializing the pull-back bundle $x^*(TM)$ over $TM$ through the frame $\EE$ defined in Subsection \ref{subsec:ortho-triv}, the Equation \eqref{eq:Sturm-manifolds} is changed into 
\begin{equation}\label{eq:Sturm-Euclidian}
\begin{cases}
-\ddt\big(\frac 1 TP(t) u'(t)+ Q(t)u(t)\big)+ \trasp{Q}(t) u'(t)+TR(t)u(t)=0, \qquad t \in (0,1)\\
\\
u(0)=Au(1), \qquad u'(0)= Au'(1).
\end{cases}	
\end{equation}
By setting   $y(t)=\frac 1 TP(t) u'(t)+Q(t)u(t)$ and $z(t)=\trasp{(y(t),u(t))}$ we finally get 
\begin{multline}\label{eq:Hamilton system}
\begin{cases}
z'(t)=JB(t)z(t), \qquad t \in [0,1]\\
z(0)=A_d z(1)
\end{cases}\quad 
\textrm{ where } \\
B(t)\=\begin{bmatrix}T P^{-1}(t) &- TP^{-1}(t)Q(t)\\-TQ(t)P^{-1}(t)& T\trasp{Q}(t)P^{-1}(t)Q(t)-TR(t) \end{bmatrix} 
\end{multline} 
and $A_d$ has been defined in Equation~\eqref{eq:Ad}.

In the standard symplectic space $(\R^{2n}, \omega)$, we denote by $J$ the standard symplectic matrix defined by $J=\begin{bmatrix} 0&-\Id\\ \Id &0\end{bmatrix}$. Thus the symplectic form $\omega$ can be  represented with respect to the Euclidean product $\langle\cdot, \cdot\rangle$ by $J$ as follows $\omega(z_1,z_2)=\langle J z_1,z_2\rangle$ for every $z_1, z_2 \in \R^{2n}$. 

Now, given $M \in \Sp(2n, \R)$, we denote by $\Gr(M)=\{(x,Mx)|x\in \R^{2n}\}$ its graph and we recall that  $\Gr(M)$ is a Lagrangian subspace of the symplectic  space $(\R^{2n} \times \R^{2n}, -\omega \times \omega)$.  
\begin{defn}\label{def:geometrical-index-NA}
	Let $x$ be a $1$-periodic solution of Equation~\eqref{eq:E-L equation-fixed},  $z_x$ be the solution of corresponding Hamiltonian equation and let us consider the path 
	\begin{equation}
	\gamma_\Phi :[0,1] \to \Sp(2n, \R) \quad \textrm{ given  by }\quad  \gamma_\Phi(t)\= A_d[\Phi^E(t)]^{-1} D \phi_H^t(z_x(0))\Phi^E(0).
	\end{equation}
	We define the {\em geometrical index of $x$\/} as follows
	\begin{equation}\label{eq:geo-na}
	\igeo(x)\=\iCLM(\Delta,\Gr(\gamma_\Phi(t)), t\in[0,1])
	\end{equation}
	where the (RHS) in Equation~\eqref{eq:definition of maslov index}  denotes  the $\iCLM$ intersection index between the Lagrangian path $t \mapsto \Gr(\gamma_\Phi(t))$ and the Lagrangian path $\Delta\=\Gr(\Id)$.
\end{defn}
Let $x$ be a $1$-periodic solution of Equation~\eqref{eq:E-L equation-fixed} and  $z_x$ be the solution of corresponding Hamiltonian equation, we can define the {\em linearized Poincaré map of $z_x$\/} as follows. 
\begin{multline}\label{eq:linearized-poincare-map}
\mathfrak P_{z_x}:  T_{x(0)}M\oplus T_{x(0)}^*M \to 
T_{x(0)}M\oplus T_{x(0)}^*M  \textrm{ is   given by 	}\\
\mathfrak P_{z_x}(\alpha_0, \delta_0)\=\bar A_d\trasp{\Big( \zeta(T),\frac 1 T\bar P(T) \Ddt \zeta(T)+ \bar Q(T)\zeta(T)\Big)}\\ \quad \textrm{ for } \quad 
\bar A_d \= \begin{bmatrix}
\bar A & 0 \\ 0 & \bar A
\end{bmatrix}
\end{multline}
where $\zeta$ is the unique vector field along $x$ such that $\zeta(0)=\alpha_0$ and  $\frac 1 T\bar P(0) \Ddt \zeta(0)+ \bar Q(0)\zeta(0)=\delta_0$. Fixed points of $\mathfrak P_{z_x}$ corresponds to periodic vector fields along $z_x$. 

By pulling back the linearized Poincaré map defined in Equation~\eqref{eq:linearized-poincare-map} through the unitary  trivialization $\Phi^\EE$ of $z_x(TT^*M)$ over $[0,1]$ we get the map 
\begin{equation}\label{eq:linearized-poincare-euclidea}
P^\EE: \R^n \oplus \R^n \to \R^n \oplus \R^n \textrm{ defined by } 
P^\EE(y_0, u_0)=A_d\trasp{\big(\frac 1 TPu'(T)+ Q u(T),  u(T)\big)}
\end{equation}
where $z(t)=\big(y(t),u(t)\big)$ is the unique solution of the Hamiltonian system given in Equation~\eqref{eq:Hamilton system} such that $z(0)=(y_0, u_0)$. 

Denoting by  $t \mapsto\psi(t)$ the fundamental solution of (linear) Hamiltonian system  given in Equation~\eqref{eq:Hamilton system}, then we get the geometrical index given in Definition \ref{def:geometrical-index-NA} reduces to  
\begin{equation}\label{eq:definition of maslov index}
\igeo(x)\=\iCLM(\Delta,\Gr(A_d\psi(t)), t\in[0,1]).
\end{equation}
Moreover, the linearized Poincaré map can be given by the symplectic matrix $A_d\psi(1)$.

By choosing a suitable coordinates and trivialization we can split $A_d\psi(t)$ into following form:
\begin{equation}\label{eq:poincare-splitting}
A_d\psi(t)=\begin{bmatrix}
1&0&0\\-tT'(h)&1&0\\0&0&P_x(t)
\end{bmatrix},
\end{equation}
where $P_x(t)$ is a path of $2(n-1)\times 2(n-1)$ symplectic matrices. It is referred to \cite[Page 104-105]{MP10}.
 \begin{defn}\label{def:spectral-stability}
	Under above notations, we term $z_x$ {\em spectrally stable\/} if the spectrum $\spec({P_{x}(1)}) \subset \U$ where $\U\subset \C$ denotes the unit circle of the complex plane. Furthermore, if $ P_{x}(1)$ is also semi-simple, then $z_x$ is termed {\em linearly stable\/}.
\end{defn}
Denote $\gamma_1(t)=\Big\{\begin{bmatrix}
1&0\\-tT'(h)&1
\end{bmatrix}, t\in[0,1]\Big\}$ and $\gamma_2(t)=\{P_x(t) \ | \ t\in[0,1]   \}$, then by \eqref{eq:special-maslov-index} we have 
\begin{equation}\label{eq:difference-maslov}
\begin{aligned}
\igeo(x)&=\iCLM(\Delta,\Gr(\gamma_1(t)), t\in[0,1])+\iCLM(\Delta,\Gr(\gamma_2(t)),t\in[0,1])\\&=\iCLM(\Delta,\Gr(\gamma_2(t)),t\in[0,1])+\left\{
\begin{aligned}
&1,\qquad \text{if}\quad T'(h)<0;\\
&0,\qquad \text{if}\quad T'(h)\ge 0.
\end{aligned}
\right.
\end{aligned}
\end{equation}

Next result is one of the two main ingredients that we need for proving our main results and in particular relates the parity of the $\iomega{1}$-index to the linear instability of the periodic orbit. 
\begin{lem}\label{lem:instability by maslov index}
The following implication holds
	\[ 
	\iCLM(\Delta,\Gr(\gamma_2(t)),t\in[0,1]) \textrm{ is odd } \Rightarrow x \textrm{ is linearly unstable }
	\]
\end{lem}
\begin{proof}
It is referred to the proof of \cite[Lemma 3.15]{PWY21}.

\end{proof}

In \cite[Equation 4.25]{PWY21} we give the precise relationship between $\igeo(x)$ and $\ispec^T(x)$:
\begin{equation}\label{eq:relation-geo-spec}
\igeo(x)=\ispec^T(x)+\dim\ker(A-I).
\end{equation}

\subsection{Proof of Main Theorem}

\begin{proof}[Proof of Theorem \ref{thm:main-instability}]
	
We prove only the (contrapositive of) the first statement in Theorem \ref{thm:main-instability}, being the others completely analogous. Thus, we aim to prove that
\[
\textrm{ if } x  \textrm{ is L-positive, orientation preserving and linearly stable } \quad \Rightarrow\quad  \ispec(x,T) +n \textrm{ is odd. }
\]
First of all, we have
\begin{equation}\label{eq:n+ispec}
\begin{aligned}
n+\ispec(x,T)&=n+(\ispec(x,T)-\ispec^T(x))+\ispec^T(x)\\
&=(n-\dim\ker(A-I))+(\ispec(x,T)-\ispec^T(x))+(\ispec^T(x)+\dim\ker(A-I))\\
&=(n-\dim\ker(A-I))+(\ispec(x,T)-\ispec^T(x))+
\igeo(x)\\
&=(n-\dim\ker(A-I))+(\ispec(x,T)-\ispec^T(x))+(\igeo(x)-\iCLM(\Delta,\Gr(\gamma_2(t))))\\
&\qquad \qquad +\iCLM(\Delta,\Gr(\gamma_2(t))).
\end{aligned}
\end{equation}

Being $x$  orientation  preserving (by assumption), then $\det A=1$ and being  $A$ also  orthogonal, then we get that 
\begin{itemize}
	\item $n \textrm{ even } \Rightarrow \dim\ker(A-\Id) \textrm{ even}$;
	\item $n \textrm{ odd } \Rightarrow \dim\ker(A-\Id) \textrm{ odd}$.
\end{itemize}
So, in both cases  we have $n-\dim\ker(A-\Id)$ is even.  

Since $x$ is $L$-Positive, then $\kappa(t)>0$. By equations~\eqref{eq:difference-spfl}-\eqref{eq:difference-maslov}, we have 
\begin{itemize}
	\item $T'(h) \ge 0  \Rightarrow \ispec(x,T)-\ispec^T(x) \textrm{ odd }$ and $\igeo(x)-\iCLM(\Delta,\Gr(\gamma_2(t))) \textrm{ even }$;
	\item $T'(h)<0  \Rightarrow \ispec(x,T)-\ispec^T(x) \textrm{ even}$ and $\igeo(x)-\iCLM(\Delta,\Gr(\gamma_2(t))) \textrm{ odd }$.
\end{itemize}
So, in both cases  we have $(\ispec(x,T)-\ispec^T(x))+(\igeo(x)-\iCLM(\Delta,\Gr(\gamma_2(t))))$ is odd. 

Now, if $x$ is linear stable, then by taking into account  Lemma~\ref{lem:instability by maslov index}, we get that $\iCLM(\Delta,\Gr(\gamma_2(t)))$ is even. Then, by Equation~\eqref{eq:n+ispec} we finally get that $n+\ispec(x,T)$ is odd. This concludes the proof. 
\end{proof}

\appendix

\appendix
%
%========================================
%========================================
\section{A symplectic excursion on the Maslov index}\label{sec:Maslov}
%========================================
%========================================

The purpose of this Section is to provide the  symplectic  preliminaries used in the paper.  In Subsection \ref{subsec:Maslovindexpath} the main 
properties  of the intersection number for curves of Lagrangian subspaces with respect to a  distinguished one are collected and the {\em (relative) Maslov index\/} is 
defined.  Our basic references are \cite{PPT04,CLM94, GPP04, RS93,MPP05,MPP07}.

%==========================
\subsection{A quick recap on the the CLM-index}\label{subsec:Maslovindexpath}
%==========================

Given a $2n$-dimensional (real) symplectic space $(V,\omega)$, a {\em 
Lagrangian 
subspace\/} of $V$ is an $n$-dimensional subspace $L \subset V$ such that $L = 
L^\omega$ where $L^\omega$ denotes the {\em symplectic orthogonal\/}, i.e. the 
orthogonal 
with respect to the symplectic structure. 
We denote by $ \Lagr= \Lagr(V,\omega)$ the {\em Lagrangian Grassmannian of 
$(V,\omega)$\/}, namely the set of all Lagrangian subspaces of $(V, \omega)$
\[
\Lagr(V,\omega)\=\Set{L \subset V| L= L^{\omega}}.
\]
It is well-known that $\Lagr(V,\omega)$ is a manifold. For each $L_0 \in \Lagr$, 
let 
\[
\Lagr^k(L_0) \= \Set{L \in \Lagr(V,\omega) | \dim\big(L \cap L_0\big) =k } 
\qquad k=0,\dots,n.
\]
Each $\Lagr^k(L_0)$ is a real compact, connected submanifold of codimension 
$k(k+1)/2$. The topological closure 
of $\Lagr^1(L_0)$  is the {\em Maslov cycle\/} that can be also described as 
follows
\[
 \Sigma(L_0)\= \bigcup_{k=1}^n \Lagr^k(L_0)
\]
The top-stratum $\Lagr^1(L_0)$ is co-oriented meaning that it has a 
transverse orientation. To be 
more precise, for each $L \in \Lagr^1(L_0)$, the path of Lagrangian subspaces 
$(-\delta, \delta) \mapsto e^{tJ} L$ cross $\Lagr^1(L_0)$ transversally, and as 
$t$ increases the path points to the transverse direction. Thus  the Maslov cycle is two-sidedly embedded in 
$\Lagr(V,\omega)$. Based on the topological properties of the Lagrangian 
Grassmannian manifold, 
it is possible to define a fixed endpoints homotopy invariant called {\em Maslov 
index\/}.

\begin{defn}\label{def:Maslov-index}
Let $L_0 \in \Lagr(V,\omega)$ and let $\ell:[0,1] \to \Lagr(V, \omega)$ be a 
continuous path. We 
define the {\em Maslov index\/} $\iCLM$ as follows:
\[
 \iCLM(L_0, \ell(t); t \in[a,b])\= \left[e^{-\varepsilon J}\, \ell(t): 
\Sigma(L_0)\right]
\]
where the right hand-side denotes the intersection number and $0 < \varepsilon 
<<1$.
\end{defn}
For further reference we refer the interested reader to \cite{CLM94} and references therein. 
\begin{rem}
 It is worth noticing that for $\varepsilon>0$ small enough, the Lagrangian 
subspaces 
 $e^{-\varepsilon J} \ell(a)$ and $e^{-\varepsilon J} \ell(b)$ are off the 
singular cycle. 
\end{rem}
One efficient way to compute the Maslov index, was introduced by authors in 
\cite{RS93} via 
crossing forms. Let $\ell$ be a $\mathscr C^1$-curve of Lagrangian subspaces 
such that 
$\ell(0)= L$ and let $W$ be a fixed Lagrangian subspace transversal to $L$. For 
$v \in L$ and 
small enough $t$, let $w(t) \in W$ be such that $v+w(t) \in \ell(t)$.  Then the 
form 
\[
 Q(v)= \dfrac{d}{dt}\Big\vert_{t=0} \omega \big(v, w(t)\big)
\]
is independent on the choice of $W$. A {\em crossing instant\/} for $\ell$ is an 
instant $t \in [a,b]$ 
such that $\ell(t)$ intersects $W$ nontrivially. At each crossing instant, we 
define the 
crossing form as 
\[
 \Gamma\big(\ell(t), W, t \big)= Q|_{\ell(t)\cap W}.
\]
A crossing is termed {\em regular\/} if the crossing form is non-degenerate. If 
$\ell$ is regular meaning that 
it has only regular crossings, then the Maslov index is equal to 
\begin{equation}\label{eq:iclm-crossings}
 \iCLM\big(W, \ell(t); t \in [a,b]\big) = \coiMor\big(\Gamma(\ell(a), W; a)\big)+ 
\sum_{a<t<b} 
 \sgn\big(\Gamma(\ell(t), W; t\big)- \iMor\big(\Gamma(\ell(b), W; b\big)
\end{equation}
where the summation runs over all crossings $t \in (a,b)$ and $\coiMor, \iMor$ 
are the dimensions  of 
the positive and negative spectral spaces, respectively and $\sgn\= 
\coiMor-\iMor$ is the  signature. 
(We refer the interested reader to \cite{LZ00} and \cite[Equation (2.15)]{HS09}). 
We close this section by 
recalling some useful 
properties of the Maslov index. \\
\begin{itemize}
\item[]{\bf Property I (Reparametrization invariance)\/}. Let $\psi:[a,b] \to 
[c,d]$ be a 
continuous and piecewise smooth function with $\psi(a)=c$ and $\psi(b)=d$, then 
\[
 \iCLM\big(W, \ell(t)\big)= \iCLM(W, \ell(\psi(t))\big). 
\]
\item[] {\bf Property II (Homotopy invariance with respect to the ends)\/}. For 
any $s \in [0,1]$, 
let $s\mapsto \ell(s,\cdot)$ be a continuous family of Lagrangian paths 
parametrised on $[a,b]$ and 
such that $\dim\big(\ell(s,a)\cap W\big)$ and $\dim\big(\ell(s,b)\cap W\big)$ 
are constants, then 
\[
 \iCLM\big(W, \ell(0,t);t \in [a,b]\big)=\iCLM\big(W, \ell(1,t); t \in 
[a,b]\big).
\]
\item[]{\bf Property III (Path additivity)\/}. If $a<c<b$, then
\[
 \iCLM\big(W, \ell(t);t \in [a,b]\big)=\iCLM\big(W, \ell(t); t \in [a,c]\big)+
 \iCLM\big(W, \ell(t); t \in [c,b]\big) 
\]
\item[]{\bf Property IV (Symplectic invariance)\/}. Let $\Phi:[a,b] \to \Sp(2n, 
\R)$. Then 
\[
 \iCLM\big(W, \ell(t);t \in [a,b]\big)= \iCLM\big(\Phi(t)W, \Phi(t)\ell(t); t 
\in [a,b]\big).
\]
\end{itemize}

For the special symplectic path $\gamma(t)=\begin{bmatrix}
M_{11}(t)&0\\M_{21}(t)&M_{22}(t)
\end{bmatrix}, t\in[0,T]$, there is a very useful formula to compute it's Maslov index \cite[Theorem 2.2]{Zhu06}. Here we only give the simplified version. Let $V$ be a subspace of $\C^{2n}$, define
\begin{equation}
V^I=\{x\in \C^{2n}\ |\ \omega(x,y)=0 \ \forall y\in V \}, W_I(V)=\{\trasp{(x,u,y,v)}\in\C^{4n}\ |\ \trasp{(x,y)}\in V^J, \trasp{(u,v)}\in V\}.
\end{equation}
Then there holds
\begin{equation}\label{eq:zhu's formula to compute maslov index}
\begin{aligned}
&\mu^{CLM}(W_I(V),Gr(\gamma(t)))\\
&=m^+(M_{1,1}(T)^*M_{2,1}(T)|_{S(T)})-m^+(M_{1,1}(0)^*M_{2,1}(0)|_{S(0)})+\dim S(0)-\dim S(T),
\end{aligned}
\end{equation}
where $m^+$ denotes the Morse positive index and $S(t)=\{x\in\C^n\ | \ \trasp{(x,M_{1,1}x)}\in V^I\}$. Please note that in our situation we take $K,R$ in  \cite[Theorem 2.2]{Zhu06} as $I$ and $V$ respectively.

Take $V=\{\trasp{(x,x)}\ | \ x\in\R\}$, then $V^I=V$ and $\Delta\=Gr(I)=W_I(V)$. Let $\gamma(t)=\Big\{\begin{bmatrix}
1&0\\tT_0&1
\end{bmatrix}, t\in[0,1]\Big\}$ where $T_0$ is a given real constant, then we have 
\begin{equation}\label{eq:special-maslov-index}
\iCLM\big(\Delta, \gamma(t);t \in [0,1]\big)=
\left\{
\begin{aligned}
&1,\qquad \text{if}\quad T_0>0;\\
&0,\qquad \text{if}\quad T_0\leq0.
\end{aligned}
\right.
\end{equation}
In fact, note that in this case we have $S(0)=S(T)=\R$ and $M_{21}(0)=0,M_{21}(T)=T_0$, then it is just the consequence of \eqref{eq:zhu's formula to compute maslov index}.

%%%%%%%%%%%%%%%%%%%%%%%%%%%%%%%%%%%%%%%%%%%%%%%%%%%%%%%%%%%%
%%%%%%%
%%%%%%%
%%%%%%%
%%%%%%%%%%%%%%%%%%%%%%%%%%%%%%%%%%%%%%%%%%%%%%%%%%%%%%%%%%%%

\subsection{The symplectic group and the Maslov-type index}

In the standard symplectic space $(\R^{2n}, \omega)$ we denote by $J$ the standard symplectic matrix defined by $J=\begin{bmatrix} 0&-\Id\\ \Id &0\end{bmatrix}$. The symplectic form $\omega$ can be  represented with respect to the Euclidean product $\langle\cdot, \cdot\rangle$ by $J$ as follows $\omega(z_1,z_2)=\langle J z_1,z_2\rangle$ for every $z_1, z_2 \in \R^{2n}$.  We consider the 1-codimensional (algebraic) subvariety 
\[
\Sp(2n, \R)^{0}\=\{M\in \Sp(2n, \R)| \det(M-\Id)=0\} \subset \Sp(2n, \R)
\] 
and let us define
\[
\Sp(2n,\R)^{*}=\Sp(2n,\R)\backslash \Sp(2n, \R)^{0}=\Sp(2n,
\R)^{+}\cup
\Sp(2n, \R)^{-}
\]
where 
\begin{multline}
\Sp(2n,\R)^+\=\{M\in \Sp(2n,\R)| \det(M-\Id)>0\} \quad \textrm{ and }\\
\Sp(2n,\R)^- \=\{M\in \Sp(2n,\R)
|\det(M-\Id)<0\}.
\end{multline}

For any $M\in \Sp(2n,\R)^{0}$, $\Sp(2n,\R)^{0}$ is
co-oriented at the point $M$
 by choosing  as positive direction the direction determined by
 $\frac{d}{dt}Me^{tJ}|_{t=0}$ with $t\ge 0$ sufficiently small.  We recall that $\Sp(2n,\R)^{+}$ and
$\Sp(2n,\R)^{-}$ are two path connected
 components of $\Sp(2n,\R)^{*}$
  which are simple connected in $\Sp(2n,\R)$. (For the proof of these facts we refer, for instance,  the interested reader to \cite[pag.58-59]{Lon02} and references therein).  Following authors in \cite[Definition 2.1]{LZ00} we start by recalling the following definition.
 \begin{defn}\label{def:Maslov-index-ok}
Let $\psi:[a,b]\rightarrow\Sp(2n,\R)$ be a continuous path. Then there exists an $\varepsilon>0$ such that for every $\theta\in[-\varepsilon,\varepsilon]\setminus\{0\}$, the matrices $\psi(a)e^{J\theta}$ and $\psi(b)e^{J\theta}$  lying both out of $\Sp(2n,\R)^0$ . We define the {\em $\iomega{1}$-index\/} or the {\em Maslov-type index\/} as follows
\begin{equation}
\iomega{1}(\psi)\=[e^{-J\varepsilon}\psi:\Sp(2n,\R)^{0}]
\end{equation}
where the (RHS) denotes the intersection number between the perturbed path $t\mapsto e^{-J\varepsilon} \psi(t)$ with the singular cycle $\Sp(2n,\R)^0$. 
\end{defn}
Through the parity of the $\iomega{1}$-index it is possible to locate  endpoints of the perturbed symplectic path  $t\mapsto e^{-J\varepsilon} \psi(t)$.  
 \begin{lem}\label{lem:parity property}{\bf (\cite[Lemma 5.2.6]{Lon02})\/}
 Let $\psi:[a,b] \to \Sp(2n, \R)$ be a continuous path.  The  following characterization holds
 \begin{itemize}
 \item[] $\iomega{1}(\psi)$ is even $\iff$ both the endpoints
 $e^{-\varepsilon J}\psi(a)$ and $e^{-\varepsilon J}\psi(b)$ lie in $\Sp(2n, \R)^+$ or
in $\Sp(2n, \R)^-$.
 \end{itemize}
 \end{lem}
We close this section with a rather technical  result which will be used in the proof of the main instability criterion. 
\begin{lem}{\bf (\cite[Lemma 3.2]{HS10}.
)\/}\label{prop:how to know in which component}
 Let $\psi:[a,b] \to \Sp(2n, \R)$ be a continuous symplectic path such that $\psi(0)$ is linearly stable.
 \begin{enumerate}
  \item  If $1 \notin \sigma\big(\psi(a)\big)$ then there exists $\varepsilon >0$
sufficiently small such that
  $\psi(s) \in \Sp(2n, \R)^+$ for $|s| \in (0, \varepsilon)$.
  \item We assume that $\dim \ker \big(\psi(a)-\Id\big)=m$ and
  $\trasp{\psi(a)}J \psi'(a)\vert_V$ is non-singular for $V\= \psi^{-1}(a) \R^{2m}$. If
  $\ind\big({\trasp{\psi(a)}J \psi'(a)\vert_V}\big)$ is even [resp. odd] then there exists $\delta>0$
  sufficiently small such that  $\psi(s) \in \Sp(2n, \R)^+$
  [resp. $\psi(s) \in \Sp(2n, \R)^-$] for $|s| \in (0, \delta)$.
  \end{enumerate}
\end{lem}
\begin{rem}
Knowing that $M \in \Sp(2n,\R)^0$, without any further information, it is not possible a priori to locate in which path connected components of $\Sp(2n, \R)^*$ is located the perturbed  matrix $e^{\pm \delta J}M$ for arbitrarily small positive  $\delta$. However if $M$ is linearly stable, we get the following result.
\end{rem} 
\begin{lem}\label{lem:linear stable is in positive side}
 Let $M\in \Sp(2n, \R) $ be a linearly stable symplectic matrix (meaning that $\spec{M} \subset \U$ and $M$ is diagonalizable). 
 Then, there exists $\delta >0$ sufficiently small such that $e^{\pm \delta J}M \in \Sp(2n,\R)^+$.
\end{lem}
\begin{proof}
Let us consider the  symplectic path pointwise defined by
 $M(\theta)\= e^{-\theta J}M$. By a direct computation we get that
\[
 \trasp{M(\theta)} J \dfrac{d}{d\theta} M(\theta)\Big\vert_{\theta=0} =
\trasp{M}M.
\]
We observe that $\trasp{M}M$ is symmetric and positive semi-definite; moreover since $M$ invertible it follows that
$\trasp{M}M$ is actually positive definite. Thus, in particular,
$\iiindex(\trasp{M}M)=0$. By invoking Lemma \ref{prop:how to know in which component} it follows
that there exists $\delta >0$ such that $M(\pm \delta) \in \Sp(2n, \R)^+ $. This concludes the proof.
\end{proof}

%%%%%%%%%%%%%%%%%%%%%%%%%%%%%%%%%%%%%%%%%%%%%%%%%%%%%%%%%%%%
%%%%%%%
%%%%%%%
%%%%%%%
%%%%%%%%%%%%%%%%%%%%%%%%%%%%%%%%%%%%%%%%%%%%%%%%%%%%%%%%%%%%

%==========================
\section{On the Spectral Flow for bounded selfadjoint Fredholm operators}\label{sec:spectral-flow}
%==========================
%
Let $\mathcal W, \mathcal H$ be  real separable Hilbert spaces with a dense 
and 
continuous inclusion $\mathcal W \hookrightarrow \mathcal H$.
\begin{note}
We denote by  
$\mathcal{B}(\mathcal W,\mathcal H)$ the Banach  space of all linear bounded 
operators (if $\mathcal W=\mathcal H$ we use the shorthand notation 
$\mathcal{B}(\mathcal H)$); by $\mathcal{B}^{sa}(\mathcal W, \mathcal H)$ we 
denote the set of all  bounded selfadjoint operators when regarded as operators 
on  $\mathcal H$ and finally $\mathcal{BF}^{sa}(\mathcal W, \mathcal H)$ 
denotes 
the set of all bounded selfadjoint Fredholm operators and we recall that an 
operator $T \in \mathcal{B}^{sa}(\mathcal W, \mathcal H)$ is Fredholm if and 
only if its kernel is finite dimensional and its image is closed. 
\end{note}
For $T \in\mathcal{B}(\mathcal W, \mathcal H)$ we recall that the {\em spectrum 
\/} of $T$ is 
\[
\sigma(T)\= \Set{\lambda \in \C| T-\lambda I \text{ is not 
invertible}}
\]
 and that $\sigma(T)$ is decomposed into the {\em essential 
spectrum\/} and the {\em discrete spectrum\/} defined respectively as 
$\sigma_{ess}(T) \= \Set{\lambda \in \C| T-\lambda I \notin 
\mathcal{BF}(\mathcal W, \mathcal H)}$ and $\sigma_{disc}(T)\= \sigma(T) 
\setminus \sigma_{ess}(T)$.
It is worth noting that $\lambda \in \sigma_{disc}(T)$ if and only if it is an 
isolated point in $\sigma(T)$ and $\dim \ker (T - \lambda I)<\infty$.

Let now $T \in \mathcal{BF}^{sa}(\mathcal W,\mathcal H)$, then either $0$ is 
not 
in $\sigma(T)$ or it is in $\sigma_{disc}(T)$ (cf. \cite[Lemma 2.1]{Wat15}), 
and, as a consequence of the Spectral Decomposition Theorem (cf. \cite[Theorem 
6.17, Chapter 
III]{Kat80}), the following orthogonal decomposition holds
\[
 \mathcal W = E_-(T) \oplus \ker T \oplus E_+(T),
\]
with the property
\[
 \sigma(T) \cap(-\infty, 0)= \sigma\left(T_{E_-(T)}\right) \textrm{ and } 
 \sigma(T) \cap(0,+\infty)= \sigma\left(T_{E_+(T)}\right).
\]
\noindent
\begin{defn}\label{def:Morseindex}
Let $T \in \mathcal{BF}^{sa}(\mathcal W,\mathcal H)$. If $\dim E_-(T)<\infty$ 
(resp.  $\dim 
E_+(T)<\infty$), 
we define its {\em Morse index\/} (resp. {\em Morse co-index\/})
as the integer denoted by $\iMor(T)$  (resp. $\coiMor(T)$) and defined as:
\[
 \iMor(T) \= \dim E_-(T)\qquad \big(\textrm{resp. } \coiMor(T)\= \dim E_+(T)\big).
\]
\end{defn}
\smallskip
The space  $\mathcal{BF}^{sa}(\mathcal H)$ was intensively 
investigated by Atiyah and Singer in \cite{AS69} \footnote{% 
Actually, in this reference, only skew-adjoint Fredholm operators were 
considered, but the case of bounded selfadjoint Fredholm operators presents no 
differences.} and the following important topological characterisation can be 
deduced.
\begin{prop}\label{thm:as69} (Atiyah-Singer, \cite{AS69})
The space $\mathcal{BF}^{sa}(\mathcal H)$ consists of three connected 
components:
\begin{itemize}
\item the {\em essentially positive\/} $
 \mathcal{BF}^{sa}_+(\mathcal H)\=\Set{T \in\mathcal{BF}^{sa}(\mathcal 
H)|\sigma_{ess}(T) 
  \subset (0,+\infty)}$; 
 \item the  {\em essentially negative\/} 
  $\mathcal{BF}^{sa}_-(\mathcal H)\=\Set{T \in\mathcal{BF}^{sa}(\mathcal 
H)|\sigma_{ess}(T) 
  \subset (-\infty,0)}$;
  \item the {\em strongly indefinite\/}
$ \mathcal{BF}^{sa}_*(\mathcal H)\=\mathcal{BF}^{sa}(\mathcal 
H)\setminus(\mathcal{BF}^{sa}_+(\mathcal H)
  \cup \mathcal{BF}^{sa}_-(\mathcal H)). $
  \end{itemize}
The spaces $\mathcal{BF}_+^{sa}(\mathcal H),\mathcal{BF}_-^{sa}(\mathcal H)$ 
are  contractible (actually convex), whereas $ \mathcal{BF}^{sa}_*(\mathcal 
H)$ is topological non-trivial; more precisely, $\pi_1(\mathcal{BF}^{sa}_*(\mathcal 
H))\simeq\Z.$
\end{prop}
\begin{rem}
By the definitions of the connected components of $\mathcal{BF}^{sa}$, we 
deduce 
that a bounded linear  operator is essentially positive if and only if it is a symmetric 
compact perturbation of a (bounded) positive definite selfadjoint operator. 
Analogous observation hold for essentially negative operators. 
\end{rem}
Even if in the strongly indefinite case the  Morse index as well as the  Morse co-index are meaningless, it is possible to define a sort of relative version, usually called {\em relative Morse index\/}. In order to fix our notation, we need to recall some definitions and basic facts. Our basic references are \cite{Abb01,Kat80}.

Let  $V, W$ be two  closed subspaces of $\mathcal H$ and let $P_V$ (resp. $ P_W$)  denote the orthogonal projection onto $V$ (resp. $W$).  $V, W$ are commensurable if the operator $P_V- P_W$ is compact.
\begin{defn}\label{def:commensurable-pairs}
 Let  $V$ and $W$ be two closed commensurable subspaces of $\mathcal H$. The {\em relative dimension\/} of $W$ with respect to $V$ is the integer 
\begin{equation}\label{eq:commensurable}
\dim(W,V)\= \dim (W \cap V^\perp)- \dim (W^\perp \cap V).	
\end{equation}
\end{defn}
\begin{rem}\label{rem:mi-serve}
	We observe that the spaces $W \cap V^\perp$ and  $W^\perp \cap V$  appearing in the (RHS) of Equation~\eqref{eq:commensurable} are finite dimensional because they are the spaces of fixed points of the compact operators $P_{V^\perp}P_W$ and $P_{W^\perp}P_V$, respectively. (For further details, cfr. \cite[pag.44]{Abb01}).
\end{rem}
\begin{prop}(\cite[Proposition 2.3.2]{Abb01})\label{thm:abbo2.3.2}
Let  $S,T \in \mathcal{BF}^{sa}(\mathcal H)$ be 
such that $S-T$ is a compact. Then the negative (resp. positive) eigenspaces are commensurable.  
\end{prop}
As direct consequence of Proposition \ref{thm:abbo2.3.2}, 
we are entitled to introduce the following definition.
\begin{defn}\label{def:relativeMorseindex}
We define the {\em relative Morse index\/} of an ordered pair of selfadjoint Fredholm operators  $S, T \in \mathcal{BF}^{sa}(\mathcal H)$ such that $S-T$ 
is compact, as the integer 
\begin{equation}\label{eq:relative-dimension-general}
 \irel(T,S)\= \dim\Big(E_-(S)\cap \big(E_+(T)\oplus E_0(T)\big)\Big)-\dim\Big(E_-(T)\cap \big(E_+(S)\oplus E_0(S)\big)\Big).
\end{equation}
\end{defn}
\begin{rem}
For further details on pairs of commensurable subspaces and Fredholm pairs, we refer the interested reader to \cite[Chapter 2]{Abb01}. Compare with \cite{ZL99} and \cite[Definition 2.1]{HS09}.
\end{rem}
If $S,T\in \mathcal{GL}^{sa}(\mathcal H)$, such that $S-T$ 
is compact then the relative Morse index given in Equation~\eqref{eq:relative-dimension-general}, reduces to   
\begin{equation}\label{eq:relative-dimension-special}
 \irel(T,S)\= \dim\big(E_-(S),E_-(T)\big).
\end{equation}
If $S$ is a compact perturbation of a non-negative definite 
operator $T$, so, it is essentially positive and $
\irel(T,S)=\iMor(S).$ 
\begin{rem}\label{rmk:segno-sf}
It is worth noticing that in the special case of Calkin equivalent positive isomorphisms, authors in \cite{FPR99}, defined the relative Morse index $I(T,S)$. However, their definition agrees with the one given in Definition \ref{def:relativeMorseindex} only up to the sign.  
\end{rem}

We are now in position to introduce the spectral flow. 

Given a  $\mathscr C^1$-path  $L:[a,b]\to\mathcal{BF}^{sa}(\mathcal W, \mathcal 
H)$, the spectral flow of $L$ counts the net number of eigenvalues crossing 0. 
\begin{defn}\label{def:crossing}
An instant $t_0 \in (a,b)$ is called a \emph{crossing instant} (or {\em 
crossing\/} for short) if $\ker  L_{t_0} \neq \{0\}$. The \emph{crossing form} 
at a crossing $t_0$ is the quadratic form defined by 
\[
 \Gamma( L, t_0): \ker  L_{t_0} \to \R, \ \ \Gamma( L, 
t_0)[u] \=\langle 
 \dot{ L}_{t_0} u, u \rangle_{\mathcal H},
\]
where we denoted by $\dot{L}_{t_0}$ the derivative of $L$ 
with respect to the parameter $t \in [a,b]$ at the point $t_0$.
A crossing is called \emph{regular}, if $\Gamma( L, t_0)$ is 
non-degenerate. If $t_0$ is a crossing instant for $L$, we refer to 
$m(t_0)$ the dimension of $\ker  L_{t_0}$.
\end{defn}

\begin{rem}
It is worth noticing that regular crossings are isolated, and hence, on a compact interval are in a finite number. Thus we are entitled to introducing the following definition.
\end{rem}

In the case of regular curve (namely a curve having only regular crossings)  we introduce the following Definition. 
\begin{defn}\label{def:new-spectralflow-def}
 Let  $L:[a,b]\to\mathcal{BF}^{sa}(\mathcal  H)$ be a $\mathscr 
C^1$-path and 
 we assume that it has only regular crossings. Then 
 \begin{equation}\label{eq:spectral-flow-crossings}
\spfl(L; [a,b])= \sum_{t \in (a,b)} \sgn \Gamma(L, t)- 
\iMor\big(\Gamma(L,a)\big)
+ \coiMor\big(\Gamma(L,b)\big),
\end{equation}
where the sum runs over all regular (and hence in a finite number) strictly contained in  $[a,b]$.
\end{defn}
\begin{rem}
It is worth noticing that, it is always possible to perturb the path $L$ for getting a regular path. Moreover, as consequence of the fixed end-points homotopy invariance of the spectral flow, the spectral flows of the original (unperturbed) and the perturbed  paths both coincide.
\end{rem}

\begin{defn}\label{def:positive-paths}
	The $\mathscr C^1$-path $L:[a,b]\ni t \mapsto L_t\in \mathcal{BF}^{sa}(\mathcal  H)$ is termed {\em positive\/} or {\em plus\/} path, if at each crossing instant $t_*$ the crossing form $\Gamma(L, t_*)$ is positive definite.  
\end{defn}
\begin{rem}
We observe that in the case of a positive path, each crossing is regular and in particular the total number of crossing instants on a compact interval is finite. Moreover the local contribution at each crossing to the spectral flow is given by the dimension of the intersection. Thus given a positive path $L$, the spectral flow is given by 
 \[
\spfl(L; [a,b])= \sum_{t \in (a,b)} \dim  \ker L(t)+ \dim  \ker L(b). 
\]
\end{rem}
\begin{defn}\label{def:admissible-paths-operators}
The path $L:[a,b]\to\mathcal{BF}^{sa}(\mathcal  H)$ is termed {\em admissible\/} provided it has invertible endpoints. 
\end{defn}
For paths of bounded self-adjoint Fredholm operators parametrized on $[a,b]$ which are compact perturbation of a fixed operator,  the spectral flow given in Definition \ref{def:new-spectralflow-def}, can be characterized as the relative Morse index of its endpoints. More precisely, the following result holds. 
\begin{prop}\label{thm:spfl-operatori-diff-rel-morse}
Let us consider the  path  $L: [a,b] \to\mathcal{BF}^{sa}(\mathcal  H)$  and we assume that for every $t \in [a,b]$, the operator  $L_t- L_a$ is compact.   Then
\begin{equation}\label{eq:equality-spfl-relmorse}
-\spfl(L; [a,b])=\irel(L_a, L_b).
\end{equation}
Moreover if $L_a$ is essentially positive, then we have 
\begin{equation}\label{eq:diff-Morse}
-\spfl(L; [a,b])
=\iMor (L_b)-\iMor(L_a)
\end{equation}
and if furthermore  $L_b$ is positive definite, then 
\begin{equation}
\spfl(L; [a,b])=\iMor(L_a).
\end{equation}
\end{prop} %\marginpar{reference for spectral flow and relative Morse in this setting (Pejsachiowicz assume always invertible endpoints)}
\begin{proof}
The proof of the equality in  Equation~\eqref{eq:equality-spfl-relmorse} is an immediate consequence of the fixed end homotopy properties of the spectral flow. For, let $\varepsilon >0 $ and let us consider the two-parameter family 
\[
L:[0,1]\times [a,b] \to \mathcal{BF}^{sa}(\mathcal  H) \textrm{ defined by } L(s,t)\= L_t+ s \,\varepsilon\, \Id. 
\]
By the homotopy property of the spectral flow, we get that 
\begin{multline}\label{eq:fff}
\spfl(L_t; t \in [a,b])\\	=  \spfl(L_a+s \varepsilon \Id, s \in [0,1]) + \spfl(L_t+\varepsilon\Id, t \in [a,b])- \spfl(L_b + s \varepsilon\Id, s \in [0,1])\\=  \spfl(L_t+\varepsilon\Id, t \in [a,b])
\end{multline}
where the last equality in  Equation~\eqref{eq:fff} is consequence if the positivity of all the involved paths.  By choosing a maybe smaller $\varepsilon>0$ the path  $t\mapsto L_t+\varepsilon\Id$ is admissible (in the sense of Definition  \ref{def:admissible-paths-operators}). The conclusion, now readily follows by applying  \cite[Proposition 3.3]{FPR99} (the minus sign appearing is due to a different choosing convention for the spectral flow. Cfr. Remark \ref{rmk:segno-sf}) to the path $t\mapsto L_t+\varepsilon\Id$. 
In order to prove the second claim, it is enough to observe that if $L_a$ is essentially positive, then $L$ is a  path entirely contained in the (path-connected component) $\mathcal{BF}_+^{sa}(\mathcal  H)$. The proof of the equality in Equation~\eqref{eq:diff-Morse} is now a direct  consequence of Equation the previous argument and \cite[Proposition 3.9]{FPR99}. The last can be deduced by Equation~\eqref{eq:diff-Morse} once observed that $\iMor(L_b)=0$. This concludes the proof. 
\end{proof}
\begin{rem}
	We observe that  a direct proof of Equation~\eqref{eq:diff-Morse} can be easily conceived as direct consequence of the homotopy properties of $\mathcal{BF}_+^{sa}(\mathcal  H)$.
\end{rem}

\begin{rem}
 We observe that the definition of spectral flow for bounded selfadjoint Fredholm operators given in Definition \ref{def:new-spectralflow-def} is slightly different from the standard definition given in literature in which only continuity is required on the regularity of the path.  (Cfr. For further details, we refer the interested reader to 
\cite{Phi96,RS95, Wat15} and 
 references therein). Actually Definition \ref{def:new-spectralflow-def}  represents an efficient way for 
 computing the spectral flow even if it requires more regularity as well as a  transversality assumption 
 (the regularity of each crossing instant). However, it is worth to mentioning that, the spectral flow is a fixed endpoints homotopy invariant and for admissible paths (meaning for paths having invertible endpoints) is a free homotopy invariant. By density arguments, we observe that a $\mathscr C^1$-path 
 always exists  in any fixed endpoints homotopy class of the original path.  
\end{rem}

\begin{rem}
 It is worth noting, as already observed by author in \cite{Wat15}, that the spectral flow can be 
 defined in the more general case of continuous 
 paths of closed unbounded selfadjoint Fredholm operators that are 
 continuous with respect to the (metric) gap-topology (cf. \cite{BLP05} and references 
 therein). However in the special case in 
 which the domain of the operators is fixed, then the closed path of unbounded 
 selfadjoint Fredholm operators can be regarded as a continuous path 
 in $\mathcal{BF}^{sa}(\mathcal W, \mathcal  H)$. Moreover  this path is also continuous 
 with respect to the aforementioned gap-metric topology.
 
 The advantage to regard the paths in  $\mathcal{BF}^{sa}(\mathcal W, \mathcal  H)$ is that the 
 theory is straightforward as in the bounded case and, clearly, it is sufficient for the applications  
 studied in the present manuscript. 
\end{rem}

\subsection{On the Spectral Flow for Fredholm quadratic forms}\label{subsec:flussospettraleforme}
%==========================
%
{{
Following authors in \cite{MPP05} we are in position to discuss the spectral  flow for bounded Fredholm quadratic forms.  
Let $(\mathcal H, \langle \cdot,\cdot \rangle)$ be a real separable Hilbert space. A continuous function $q : \mathcal H \to  \R$  is called a {\em quadratic form \/} provided that there is a symmetric bilinear form  $b_q : \mathcal H \times \mathcal H \to \R$ such that $q(u) = b_q(u,u)$ for all $u \in \mathcal H$.
 The bilinear form $b_q$ is determined by $q$ through the following polarization identity:
 \[
 b_q(u,v)=\dfrac14\{ q(u+v)-q(u-v)\} \textrm{ for all } u,v \in \mathcal H. 
 \]
This implies, in particular, that the bilinear form $b_q$  is also continuous. Let us denote by $\mathcal Q(\mathcal H)$  the set of all bounded quadratic  forms on $\mathcal H$ and we observe that $\mathcal Q(\mathcal H)$ is a Banach space with the norm defined by
\[ 
\norm{q} \= \sup_{\norm{u}\leq 1} |q(u)| \textrm{ for all }  q \in \mathcal Q(\mathcal H).
\]
\begin{lem}
	Let $q\in \mathcal Q(\mathcal H)$. Then there exists a unique self–adjoint operator $T \in \mathcal B(\mathcal H)$ called the representation of $q$ with respect to $\langle \cdot,\cdot\rangle $ with the property that
	    \begin{equation}\label{eq:2.2}
	     q(u) =\langle T\,u, u \rangle  \textrm{ for all  } u \in  \mathcal H
	     \end{equation}
	     Moreover 
	     \begin{equation}\label{eq:2.3}
	     b_q(u,v)=\langle T\,u, v \rangle  \textrm{ for all  } u, v \in \mathcal H.
	     \end{equation}
	    \end{lem}
\begin{proof}
We start to observe that, from the polarization identity, it follows immediately that an operator satisfies Equation~\eqref{eq:2.2} if and only if it satisfies Equation 
	\eqref{eq:2.3}. Fix $u \in \mathcal H$. Since $b_q$ is continuous and bilinear, then the map  $v \mapsto b_q(u,v)$ is linear and continuous on $\mathcal H$. By the Riesz-Fréchét Representation Theorem, there is an element that we denote by $Tu$ having the property that $b_q(u,v)=\langle T\,u, v \rangle$ for all $v \in \mathcal H$. Since $b_q$ is symmetric, bilinear and continuous we conclude that the operator $T$ belongs to $\mathcal B(\mathcal H)$ and it is symmetric with respect to $ \langle \cdot,\cdot \rangle$.
 \end{proof}
 A quadratic form $q \in \mathcal Q(\mathcal H)$ is called {\em non-degenerate\/} provided that 
\[
b_q(u,v)=0 \textrm{ for all } v \in H \Rightarrow u=0.
\]
 \begin{defn}\label{def:Fqf}
 Under above notation, a quadratic form $q:\mathcal H\to \R$ is termed a \emph{Fredholm quadratic form\/}  if the operator $T$  given in Equation~\eqref{eq:2.2} (representing $q$ with respect to the scalar product of $\mathcal H$) is Fredholm. 
\end{defn}

\begin{rem}
The set $\mathcal{Q_F}(\mathcal H)$ is an open subset of $\mathcal Q(\mathcal H)$. We observe that the operator representing of a non-degenerate quadratic form has trivial kernel and therefore if the form is also Fredholm then such a representation belongs to $\GL(\mathcal H)$.
\end{rem}
\begin{lem}\label{thm:stable-perturbation}
	A quadratic form on $\mathcal H$ is weakly continuous if and only  its representation is a compact operator in $\mathcal H$.
\end{lem}
\begin{proof}
For the proof, we refer the interested reader, to \cite[Appendix B]{BJP14}. 
\end{proof}
We define a quadratic form $q : \mathcal H \to \R$ to be positive definite provided that, there exists $c>0$ such that 
\[
q(u) \geq c \norm{u}^2, \textrm{ for all } u \in \mathcal H.
\]
We are now entitled to  introduce the following definition which will be crucial in the whole paper. 
\begin{defn}\label{def:essentially-positive}
	A quadratic form $q : \mathcal H \to \R$ is termed {\em essentially positive\/} and we write $q\in \mathcal{Q_F}^+(\mathcal H)$, provided it is a weakly continuous 
perturbation of a positive definite quadratic (Fredholm) form.\end{defn}
It turns out that if $q\in 
\mathcal{Q_F}^+(\mathcal H)$ then its representation is and essentially  positive  selfadjoint Fredholm operator. 
\begin{rem} 
We observe that as consequence of Proposition \ref {thm:as69}, the set $\mathcal{Q_F}^+(\mathcal H)$ 	is contractible.
\end{rem}
\begin{defn}\label{def:Morse-essentially-positive}
	The {\em Morse index\/}, denoted by $\iMor(q)$, of an essentially positive Fredholm quadratic form $q: \mathcal H \to \R$ is defined as the Morse index of its representation. Thus, in symbol
	\[
	\iMor(q)= \dim E_-(T)
	\]
	where $T$ is the representation of $q$ and $E_-(T)$ denotes the negative spectral space of $T$.
\end{defn}

We are now in position to introduce the notion of {\em spectral flow\/} for (continuous) path of Fredholm quadratic forms. 
\begin{defn}\label{def:sfquadratic}
Let $q:[a,b]\rightarrow \mathcal {Q_F}(\mathcal H)$ be a continuous   path.
We define  the \emph{spectral flow of $q$}  as the  spectral 
flow of the continuous path induced by its representation; namely
\[
\spfl(q(t); t \in[a,b]) \= \spfl(L_t; t \in [a,b])
\]
where we denoted by $L_t $ the representation of q(t),  namely   $q(t)= \langle L_t \cdot, \cdot \rangle_{\mathcal H}$.
\end{defn}
Following authors in \cite{MPP05}, if a path $q : [a, b] \to \mathcal {Q_F}(\mathcal H)$ is differentiable at $t$ then the derivative $\dot q(t)$ is also a quadratic form. We shall say that the instant $t$ is a crossing point if $\ker b_q(t)\neq \{0\}$, and We shall say that the crossing instant $t$ is regular if the crossing form $\Gamma(q, t)$, defined as the restriction of the derivative $\dot q(t)$ to the subspace $\ker b_q(t)$, is non-degenerate. We observe that regular crossing points are isolated and that the property of having only regular crossing forms is generic for paths in $\mathcal {Q_F}(\mathcal H)$. 
As direct consequence of Definition \ref{def:sfquadratic} and Definition \ref{def:new-spectralflow-def}, if $q:[a,b] \to \mathcal {Q_F}(\mathcal H)$ is a $\mathscr C^1$-path of Fredholm quadratic forms having only regular crossings,  we infer that  
\begin{equation}\label{eq:sf-forme-crossings}
	\spfl(q(t); t \in[a,b])=\sum_{t \in (a,b)} \sgn \big(\Gamma(q(t))\big)- 
\iMor\big(\Gamma(q(a))\big)
+ \coiMor\big(\Gamma(q(b))\big) 
\end{equation} 
where the sum runs over all regular (and hence in a finite number) strictly contained in  $[a,b]$.
\begin{defn}\label{def:admissible-path-forms}
A path $q :[a, b] \to \mathcal {Q_F}(\mathcal H)$ is called {\em admissible\/} provided that it has non-degenerate endpoints.
\end{defn}
Thus by Definition \ref{def:admissible-path-forms} and Equation~\eqref{eq:sf-forme-crossings}, we get that for an admissible $\mathscr C^1$-path in  $\mathcal {Q_F}(\mathcal H)$ having only regular crossings, the spectral reduces to 
\begin{equation}\label{eq:sf-forme-crossings-admi}
	\spfl(q(t); t \in[a,b])=\sum_{t \in (a,b)} \sgn \big(\Gamma(q(t))\big)
\end{equation}
where the sum runs over all regular (and hence in a finite number) strictly contained in  $[a,b]$.
(Cfr. \cite[Theorem 4.1]{FPR99}). 
\begin{prop}\label{thm:sf-differenza-morse}
Let $q:[a,b] \to 	 \mathcal {Q_F}(\mathcal H)$ be a path of essentially positive Fredholm quadratic forms. Then 
\begin{equation} \label{eq:sf-diffMorse}
\spfl(q(t);t\in[a,b]) = \iMor\big(q(a)\big)-\iMor\big(q(b)\big).
\end{equation}
Furthermore, if $q(b)$ is positive definite, then 
$\spfl(q(t);t \in [a,b]) = \iMor(q(a))$. 
\end{prop}
\begin{proof}
Let $T:[a,b] \to \mathcal{BF}^{sa}( \mathcal  H)$ be the path of selfadjoint Fredholm operators that represents the path of Fredholm  quadratic forms $q$. As already observed, for each $t \in [a,b]$, each  $T(t)$ is an essentially positive Fredholm operator. The result follows by Proposition \ref{thm:spfl-operatori-diff-rel-morse} to the path $t \mapsto T_t$. \\
The second claim follows readily by observing that if $q(b)$ is essentially positive and non-degenerate, then it is positive and hence the Morse index of its representation vanishes.
This concludes the proof.
\end{proof}
\begin{defn}\label{def:generalized-Fredholm}
A {\em generalized   family of Fredholm quadratic forms\/}
parameterized by an interval is a smooth function   $q\colon
{\mathcal H}\to \R ,$ where ${\mathcal H}$ is a Hilbert bundle
over $[a,b]$ and $q$ is such that its restriction $q_t$  to the
fiber ${\mathcal H}_t$  over $t$ is a Fredholm quadratic form. If
$q_a$ and $q_b$ are non  degenerate, we define the spectral flow
$\spfl(q) = \spfl (q,[a,b])$  of such a family $q$ by choosing a
trivialization $$M \colon [a,b] \times {\mathcal H}_a \to {\mathcal
H}$$ and defining
\begin{equation} \label{sflow2}
\spfl(q) = \spfl (\tilde{q},[a,b])
\end{equation}
where $\tilde{q}(t)u =q_t (M_tu).$
\end{defn}
It follows from cogredience property that the right hand side of Equation~\eqref{sflow2} is independent of the choice of the trivialization.
Moreover all of the above properties hold true in this more
general case, including the calculation of the spectral flow
through a non degenerate crossing point,  provided we substitute the usual derivative with
the intrinsic derivative of a bundle map.

% % % % % ================================================================

\vspace{1cm}
\noindent
\textsc{Prof. Alessandro Portaluri}\\
DISAFA\\
Università degli Studi di Torino\\
Largo Paolo Braccini 2 \\
10095 Grugliasco, Torino\\
Italy\\
E-mail: \email{alessandro.portaluri@unito.it}

\vspace{1cm}
\noindent
\textsc{Prof. Li Wu}\\
Department of Mathematics\\
Shandong University\\
Jinan, Shandong, 250100\\
The People's Republic of China \\
China\\
E-mail:\email{vvvli@sdu.edu.cn }

\vspace{1cm}
\noindent
\textsc{Dr. Ran Yang}\\
School of Science\\
East China  University of Technology\\
Nanchang, Jiangxi, 330013\\
The People's Republic of China \\
China\\
E-mail:\email{yangran2019@ecit.cn}
\end{document}